\documentclass[11pt,a4paper]{article}

\usepackage{titlesec,titletoc}

\usepackage[margin=0.75in]{geometry}
\usepackage{url}
\usepackage{fancyhdr}
\usepackage{setspace} 
\usepackage{appendix} 
\usepackage{enumerate}
\usepackage{calc}
\usepackage{xspace}
\usepackage{caption}
\usepackage{subcaption}
\usepackage{placeins}  
\usepackage{pdflscape} 
\usepackage{longtable}
\usepackage{xcolor,colortbl,booktabs}
\usepackage{parskip}  
\usepackage{authblk}  
\usepackage{bold-extra} 
\makeatletter
\def\thm@space@setup{%
\thm@preskip=\parskip \thm@postskip=0pt
}
\makeatother

\usepackage{amsmath}
\usepackage{amssymb}
\usepackage{amsthm}
\usepackage{mathtools}

\usepackage[
    backend=biber,
    style=ieee,
    natbib=true,
    isbn=false,
    url=false, 
    doi=false,
    eprint=false,
]{biblatex}
\usepackage{graphicx}
\graphicspath{ {./figures/} }

\newcommand{\mbf}{\mathbf}
\newcommand{\mcf}{\mathcal}
\newcommand{\mbb}{\mathbb}

\newcommand{\vLHSCB}{$\nu$--LHSCB\xspace}

\newcommand{\tpose}{^{\!\top\!}}

\renewcommand{\Re}{\mbb{R}}

\newcommand{\eqdef}{:=}

\newcommand{\norm}[1]{\left\| #1 \right\|}
\newcommand{\set}[2]{\left\{ #1\ \left| \ #2 \right. \right\}}

\newcommand{\interior}{\mathop{\operatorname{int}}}
\newcommand{\trace}{\mathop{\operatorname{tr}}}

\newcommand{\svec}{{\operatorname{svec}}}
\newcommand{\mat}{{\operatorname{mat}}}
\newcommand{\innerprod}[2]{{#1}\tpose{#2}}
\newcommand{\Cpp}{C\texttt{++}\xspace}
\newcommand{\gradient}[1]{\nabla{#1}}

\newtheorem{thm}{Thoerem}[section]
\newtheorem{lem}[thm]{Lemma}

\newtheorem{definition}[thm]{Definition}

\theoremstyle{remark}

\newcommand{\OptProblem}[5][]{
  \begin{align}\label{#1}
    \begin{array}{rl}
    #2\limits_{#3}&\hspace{-0ex}#4\vspace{1ex}\\
    \text{subject to:}&\hspace{-1ex}
    \begin{gathered}[t]
      #5
    \end{gathered}
  \end{array}
\end{align}
}

\newcommand{\MinProblem}[4][]
{\OptProblem[#1]{\min}{#2}{#3}{#4}}

\newcommand{\MaxProblem}[4][]
{\OptProblem[#1]{\max}{#2}{#3}{#4}}

\newcommand{\UnconstrainedOptProblem}[4][]{
  \begin{align}\label{#1}
    \begin{array}{rl}
    #2\limits_{#3}&\hspace{-0ex}#4\vspace{1ex}
  \end{array}
\end{align}
}

\newcommand{\UnconstrainedMinProblem}[3][]
{\UnconstrainedOptProblem[#1]{\min}{#2}{#3}}

\newcommand{\coneNN}[1]{\mathbb{R}^{#1}_{+}}
\newcommand{\coneSOCP}[1]{\mathcal{Q}_{#1}}
\newcommand{\coneSDP}[1]{\mathcal{S}_{#1}}
\newcommand{\coneEXP}{\mathcal{K}_{\textrm{exp}}}
\newcommand{\conePOW}{\mathcal{K}_{\textrm{pow}}}

\usepackage{xcolor,calc}

\newcommand{\detailtablecaption}{FOO}

\definecolor{BenchHighlight}{rgb}{0.8,0.8,0.9}
\newcommand{\winner}{\cellcolor{BenchHighlight}}

\title{\LARGE  {\bfseries Clarabel}: \textbf{An interior-point solver for conic programs with quadratic objectives}}
\author{Paul J.\ Goulart, Yuwen Chen}  

\affil{Department of Engineering Science, University of Oxford,
Oxford, UK}

\date{May 21, 2024}
\begin{document} 
\maketitle

\begin{abstract}
We present a general-purpose interior-point solver for convex optimization problems with conic constraints.   Our method is based on a homogeneous embedding method originally developed for general monotone complementarity problems and more recently applied to operator splitting methods, and here specialized to an interior-point method for problems with quadratic objectives.   We allow for a variety of standard symmetric and non-symmetric cones, and provide support for chordal decomposition methods in the case of semidefinite cones.   We describe the implementation of this method in the open-source solver \textsc{Clarabel}, and provide a detailed numerical evaluation of its performance versus several state-of-the-art solvers on a wide range of standard benchmarks problems.  \textsc{Clarabel} is faster and more robust than competing commercial and open-source solvers across a range of test sets, with a particularly large performance advantage for problems with quadratic objectives. \textsc{Clarabel} is currently distributed as a standard solver for the Python \textsc{CVXPY} optimization suite.  
\end{abstract}



\section{Introduction}

We consider throughout the following convex conic optimization problem:
\MinProblem[eqn:primal]{x, s}{\frac{1}{2}x\tpose P x + q\tpose x}{
\begin{aligned}[t]
&Ax + s = b \\[1ex]
&s \in \mcf{K},
\end{aligned} \tag{$\mcf P$}
}
with decision variables $x\in \Re^n$ and $s\in \Re^m$, and problem data
$A\in \Re^{m\times n}$, $b\in \Re^m$, $q \in \Re^n$ and $P \in \Re^{n\times n}$.   We assume that $P$ is symmetric and positive semidefinite (possibly zero) and that the set $\mcf{K}$ is a closed and convex cone. We will denote the optimal value of this problem as $p^*$ and an optimizer (when it exists) as $(x^*,s^*)$.

It can be shown that the problem dual to \ref{eqn:primal} is
\MaxProblem[eqn:dual]{x,z} {-\frac{1}{2}x\tpose P x - b\tpose z}{
\begin{aligned}[t]
  &P x  +  A\tpose z = -q, \\[1ex]
  &z \in \mcf{K}^*
\end{aligned} \tag{$\mcf{D}$}
}
where $\mathcal K^*$ is the dual cone of $\mathcal K$.   We will denote its optimal value as $d^*$ and an optimizer (when it exists) as $(x^*,z^*)$.    We will assume throughout that strong duality holds between \ref{eqn:primal} and \ref{eqn:dual}, i.e.\ that $p^* = d^*$.

The problem \ref{eqn:primal} is very general and can model most convex optimization problems of practical interest.  Such problems appear in a huge range of applications including constrained optimal control \cite{Borrelli:MPCbook} and moving horizon estimation \cite{Allan2019:MHE}, limit analysis of engineering structures \cite{Martin:LowerSOCP,Martin:UpperSOCP} and fluid flows \cite{Goulart:FluidSOS}, image processing \cite{Combettes:ImageRecovery}, support vector machines \cite{Cortes:SVM:1995}, lasso problems \cite{Tibshirani:1996,Candes:2008}, circuit design \cite{Boyd:GeoProg}, portfolio optimization \cite{cornuejols2006, JOFI:JOFI1525}, and many others.   
The numerical solution of convex problems in the form \ref{eqn:primal} also underpins many nonconvex optimization methods, including those based on sequential quadratic programming \cite[Ch.~18]{NW06}\cite{boggs:tolle:SQP:1995} and branch-and-bound methods within mixed-integer methods \cite{Bertsimas:IntegerOpt,BandB:Survey}. 

\subsection{Solution methods}
Considerable research effort has focused on the development of solution techniques for \ref{eqn:primal} over many decades.   If the cone $\mathcal{K}$ is the positive orthant, then \ref{eqn:primal} represents a {linear program} (LP) if $P = 0$ or a quadratic program (QP) otherwise.   This case in particular has a long history, dating to the landmark work of Kantorovich \cite{Kantorovich:1939} and Koopmans \cite{Koopmans:1942, Koopmans:1949} in the LP case, and Frank and Wolfe in the QP case \cite{Frank:1956}.   Of particular significance was the development of numerical solution methods for such problems, and in particular \emph{simplex} or active-set methods for the solution of LP \cite{Dantzig:1963} and QP \cite{Wolfe:1959} instances.  Active set methods still form the basis of many modern LP and QP solvers, including commercial solvers such as Gurobi \cite{gurobi} and open-source solvers such as qpOASES \cite{qpOASES} and HiGHS \cite{HiGHS}.   

However, since active set methods are based on the iterative updating of a basis of active constraints and rely heavily on the polyhedral structure of systems of linear equalities and inequalities, they do not generalize well to the case of non-polyhedral constraints.    Recent advances for more general \emph{conic programs}, i.e.\ optimization problems where the set $\mathcal{K}$ is a more general convex cone, have therefore focussed instead on either operator splitting methods such as the alternating direction method of multipliers (ADMM) \cite{boyd2011distributed} or on interior-point methods \cite{Wright97}.  ADMM methods in particular are attractive in cases where only moderate accuracy is required, since ADMM is known to be prone to poor `tail convergence' at high precision \cite{He:Yuan:2012}.   Nevertheless, ADMM based methods are widely used in practice for both their scalability and simplicity of implementation and form the basis of several open-source solvers for \ref{eqn:primal}, including the open source solvers OSQP \cite{OSQP}, SCS \cite{SCS} and COSMO \cite{COSMO}.
 
\paragraph{Primal-dual interior-point methods}  In this work we focus instead on primal-dual {interior-point} methods \cite{nemirovski:todd:2008}, since such methods are known to perform better than operator splitting methods at high precision and, unlike active set methods, can support general conic constraints.   A significant milestone in the development of solution algorithms for \ref{eqn:primal} was Karmarkar's method \cite{Karmarkar:1984}, which provided the first practically useful polynomial-time solution method for linear programs,
followed by the {path following} method of Renegar \cite{Renegar:1988}.   

An advantage of interior-point methods is that they can be extended to the case of general conic constraints $\mathcal{K}$ in \ref{eqn:primal}.   Of particular importance in the development of efficient conic interior-point solvers was the development of solvers based on cones and associated barrier functions that have the \emph{self-scaling} property~\cite{Nesterov97,Nesterov98}.   The family of self-scaled cones include those that are generally considered the most important constraint types in practice, i.e.\ the nonnegative orthant, second-order cone and positive semidefinite (PSD) cone.   Conic problems defined over self-scaled cones can be solved efficiently using primal-dual interior-point methods based on the \emph{Nesterov-Todd (NT) scaling} strategy~\cite{Nesterov97,Nesterov98}.   

Nonsymmetric cones that do \emph{not} have the self-scaling property have been the subject of increasing recent interest for interior-point methods~\cite{Chares09,Nesterov12}. The exponential cone and the power cone are the two most studied and are supported in commercial solvers such as Mosek~\cite{Dahl21}.  
In certain cases it is also attractive to formulate problems in terms of more exotic nonsymmetric cones that can be solved more efficiently by exploiting their special structure~\cite{Papp19,Hypatia,DDS}.   Since the NT scaling method does not apply to nonsymmetric cones, alternative strategies have been proposed based on either primal or dual iterates~\cite{Skajaa15,Serrano15,Hypatia} or both primal and dual iterates together~\cite{Dahl21,DDS}.

Alongside these developments, a further innovation was the development of the \emph{homogeneous self-dual embedding} (HSDE)~\cite{Ye94}.  This method embeds a conic program with linear objectives into a slightly larger feasibility problem, which is \emph{always} feasible regardless of the feasibility of the original problem.  Once a feasible point is found, it can be used to construct either an optimizer for the original problem or a primal or dual certificate of infeasibility.   
The HSDE is at the heart of many modern interior-point solvers including CVXOPT \cite{cvxopt}, ECOS \cite{ecos}, SCS \cite{SCS} and MOSEK \cite{MOSEK}.   A downside of this approach is that it does not naturally support problems with quadratic objectives, since the standard embedding method applied to such cases does not produce a self-dual problem.  

\paragraph{Monotone complementarity problems (MCPs)}

Extensions of the HSDE method were quickly developed for the linear complementary problem (LCP)~\cite{Ye97}, and later for the more general monotone complementary problem (MCP)~\cite{Andersen1999} with linear constraints.    The homogeneous embedding for MCPs was also extended to problems defined over more general symmetric cone constraints~\cite{Yoshise07,Yoshise07-a}.    

The practical performance of the homogeneous embedding in the case of convex quadratically constrained quadratic programming (QCQP) problems was investigated in~\cite{Meszaros15}.   More recently, \cite{ODonoghue2021} proposed a maximal monotone operator specialized for the LCP based on the monotone operator in~\cite{Andersen1999} to produce a first-order solver for problems with convex quadratic objectives and general conic constraints.   

\section{Algorithmic approach}

In this paper we adopt the approach of \cite{Meszaros15,ODonoghue2021} by specializing the homogeneous embedding method of \cite{Andersen1999} to convex problems with quadratic objectives.  We will use this embedding as the basis for a general purpose primal-dual interior-point solver that handles both symmetric and nonsymmetric cone constraints.   

The standard Karush-Kuhn-Tucker (KKT) conditions for problem \ref{eqn:primal} are
\begin{equation}\label{eqn:KKT}
\begin{aligned}
  Ax + s &= b \\
  Px + A\tpose z &= -q\\
  s\tpose z &= 0\\
    (s,z) &\in \mcf{K} \times \mcf{K}^*.
\end{aligned}
\end{equation}
We define the duality gap $\gamma$ as
\[
\begin{aligned}
\gamma \eqdef& ~ \left(\frac{1}{2}x\tpose P x + q\tpose x\right) - \left(-\frac{1}{2}x\tpose Px - b\tpose z\right) \\[1ex]
       =& ~x\tpose P x + q\tpose x + b\tpose z = s\tpose z,
\end{aligned}
\]
where the final equality follows from substitution in the KKT conditions \eqref{eqn:KKT}.  The duality gap $\gamma$ is nonnegative for any feasible point since $(s,z)$  are constrained to the dual pair of cones $(\mathcal{K},\mathcal{K}^*)$, and is zero at an optimal point $(x^*,s^*,z^*)$ when strong duality holds.   

Taking problems \ref{eqn:primal} and \ref{eqn:dual} together with an objective minimizing the duality gap, we can rewrite our problem in the form 
\MinProblem[eqn:PD_no_embedding]{
    (x, s, z)
  }{s\tpose z}%
  {\begin{aligned}[t]
    &x\tpose P x + q\tpose x + b\tpose z = 0 \\[1ex]
  &P x + A\tpose z + q = 0 \\[1ex]
  &A x + s - b = 0 \\[1ex]
  &(s, z) \in \mcf{K} \times \mcf{K^*}.
  \end{aligned} 
  }

  This problem is of course infeasible whenever either \ref{eqn:primal} or \ref{eqn:dual} is infeasible or the pair is not strongly dual.

  If we define nonnegative scalars $(\tau,\kappa)$ and apply a change of variables $x \to x/\tau$, $z \to z/\tau$ and $s\to s/\tau$, then we can rewrite the feasibility problem~\eqref{eqn:PD_no_embedding} as 
  \MinProblem[eqn:PD_homogeneous]{
  (x, s, z, \tau, \kappa)
  }{\innerprod{s}{z} + \tau\kappa}%
  {\begin{aligned}[t]
    &\frac{1}{\tau}x\tpose P x + q\tpose x + b\tpose z = -\kappa \\[1ex]
  &Px + A\tpose z + q\tau = 0 \\[1ex]
  &Ax + s - b\tau = 0 \\[1ex]
  &(s,z,\tau,\kappa) \in \mcf{K} \times \mcf{K^*} \times \Re_+ \times \Re_+.
  \end{aligned}  \tag{$\mcf H$}
  }

  The problem~\ref{eqn:PD_homogeneous} is equivalent to the widely-used homogeneous self-dual embedding (HSDE) \cite{Nesterov1999} in the special case $P = 0$.   The most common approach to solving \ref{eqn:primal} is therefore to eliminate the quadratic term in the objective function, replacing it with an epigraphical upper bound and an additional second-order cone constraint in the objective \cite{Lobo:1998}.   This amounts to rewriting \ref{eqn:primal} as 

\MinProblem[eqn:primal_linear]{x, s, w}%
{w + q\tpose x}%
{
  \begin{aligned}[t]
    &Ax + s = b \\[1ex]
    &\norm{\begin{bmatrix}
    2P^\frac{1}{2}x \\ w - 2
    \end{bmatrix}}_2 \le w + 2 \\[1ex]
    &s \in \mcf{K}, ~w \ge 0,
    \end{aligned} 
}
  
which is a conic optimization problem with a purely linear cost and an additional second-order cone constraint.

 If we instead keep the quadratic term in the objective then~\ref{eqn:PD_homogeneous} remains homogeneous but is no longer self-dual, in addition to featuring the seemingly awkward $\frac{1}{\tau}x\tpose P x$ term in the first equality.  Our approach will nevertheless amount to a direct solution of \ref{eqn:PD_homogeneous} using a primal-dual type interior-point method, and we will show that significant performance improvements are possible with this approach in many cases relative to standard HSDE-based interior-point methods.

\bigskip 

\subsection{Existence of solutions and infeasibility detection}\label{sec:infeasibility}  

Before proceeding to the details of our algorithm, we address briefly the existence and interpretation of solutions to \ref{eqn:PD_homogeneous}.
As in the case of the HSDE, an advantage of our reformulation is that solutions will produce either a solution to the original problem \ref{eqn:primal} or \ref{eqn:dual}, or (asymptotically) a certificate of either primal or dual infeasibility.      We first define the following sets of infeasibility certificates for \ref{eqn:primal} and \ref{eqn:dual}:
\begin{align*}
  \mathbb{P} &\eqdef \set{x}{~Px = 0,\,Ax \in \mcf{K},\,\,\,\innerprod{q}{x} < 0} \\ 
  \mathbb{D} &\eqdef \set{z}{ A\tpose z = 0,\,~~\,z\in \mcf{K}^*,\, \innerprod{b}{z} < 0}.
\end{align*}
The problem \ref{eqn:primal} is infeasible if and only if there exists some $\bar x \in \mathbb{D}$, while the problem \ref{eqn:dual} is infeasible if and only if there exists some $\bar z \in \mathbb{P}$.   Similar conditions have appeared in \cite{COSMO,Banjac2019,ODonoghue2021} in the context of conic optimization with quadratic objectives.


Solving the problem \ref{eqn:PD_homogeneous} amounts to finding a root of the nonlinear system of equations 
\begin{equation}\label{eqn:root_finding_G}
G(x,z,s,\tau,\kappa) := 
\begin{bmatrix}
  0\\s\\ \kappa 
\end{bmatrix}
-
\begin{bmatrix}
  \hphantom{+}P  & A\tpose & q \\
  -A & 0 & b\\
  -q\tpose & -b\tpose & 0
\end{bmatrix}
\begin{bmatrix}
  x \\ z \\ \tau 
\end{bmatrix}
 + 
 \begin{bmatrix}
  0 \\ 0 \\
  \frac{1}{\tau}x\tpose P x
 \end{bmatrix} = 0
\end{equation}
subject to the conic constraint 
\[
  (x,z,s,\tau,\kappa) \in \mathcal{C} :=  \Re^n \times \mathcal{K} \times \mathcal{K}^*\times\Re_+ \times \Re_+.
\]
These conditions are identical to those in \cite[(4.1)]{ODonoghue2021}, where they were developed in terms of specialization of the homogeneous embedding of monotone complementary problems developed in \cite{Andersen1999} to an LCP.

We can define a more compact notation $v := (x,z,s,\tau,\kappa)$ for convenience, so that \ref{eqn:PD_homogeneous} is equivalent to 
\MinProblem[eqn:PD_homogeneous_compact]{v}
{\innerprod{s}{z} + \tau\kappa}%
{\begin{aligned}[t]
G(v) = 0,\quad v \in \mathcal{C}.
\end{aligned} 
}
As in \cite{Andersen1999}, we say that the problem \ref{eqn:PD_homogeneous} (equivalently \eqref{eqn:PD_homogeneous_compact}) is \emph{asymptotically feasible} if there exists a bounded sequence $\{v^k\} \subset \mathcal{C}$, $k = 1,2,\dots,$ such that 
\[
\lim_{k \to \infty} G(v^k) = 0,
\]  
and call any limit point $\hat v := (\hat x,\hat z,\hat s,\hat \tau,\hat \kappa)$ of such a sequence an \emph{asymptotically feasible point}.   We will call any such point with ${\innerprod{\hat s}{\hat z} + \hat \tau\hat \kappa} = 0$ an \emph{(asymptotically) complementary} or optimal solution.  

We can now develop some basic results about solutions to the problem \ref{eqn:PD_homogeneous} in the spirit of \cite[Thm.\ 1]{Andersen1999}:

\begin{lem}
The problem \ref{eqn:PD_homogeneous} is (asymptotically) feasible and every (asymptotically) feasible point is (asymptotically) complementary.    
\end{lem}
\begin{proof} 
We can construct an asymptotically feasible sequence by defining $x^k = (\frac{1}{2})^k\mathbf{1}\frac{1}{\sqrt{n}}$, $s^k = (\frac{1}{2})^ke_\mathcal{K}$, $z^k = (\frac{1}{2})^ke_{\mathcal{K}^*}$, $\tau^k = (\frac{1}{2})^k$ and $\kappa^k = (\frac{1}{2})^k$, where $(e_\mathcal{K},e_{\mathcal{K}^*})$ is any vector pair in  $\mathcal{K} \times \mathcal{K}^*$.   It is then straightforward to show that $G(v^k) \to 0$, noting in particular that the sequence
\[
  \frac{1}{\tau^k}(x^k)\tpose Px^k \le ({\tfrac{1}{2}})^k\bar{\sigma}(P) \to 0
\]
where $\bar\sigma({P})$ is the maximum singular value of $P$,
and is lower bounded by zero since $P$ is positive semidefinite.   To show the second part, observe that 
\[
 { 
\begin{bmatrix}
  x \\ z \\ \tau 
\end{bmatrix}}  \tpose G(v) = s\tpose z + \tau\kappa
\]
for any $v \in \mathcal{C}$, so $(s^k)^T(z^k) + \tau^k\kappa^k \to 0$ since $G(v^k) \to 0$ and all sequences are bounded.
\end{proof}

\begin{lem}
  Suppose that $v^* := (x^*,z^*,s^*,\tau^*,\kappa^*)$ is an asymptotically complementary solution to \ref{eqn:PD_homogeneous}.   Then:
  \begin{enumerate}[{i)}]
    \item If $\tau^* > 0$ then $(x^*/\tau^*,s^*/\tau^*)$ is an optimal solution to \ref{eqn:primal} and $(x^*/\tau^*,z^*/\tau^*)$ is an optimal solution to \ref{eqn:dual}.
    \item If $\kappa^* > 0$ then at least one of the following holds:
    \begin{itemize}
    \item \ref{eqn:primal} is infeasible and $z^* \in \mathbb{D}$.
    \item \ref{eqn:dual} is infeasible and $x^* \in \mathbb{P}$.
    \end{itemize}
  \end{enumerate}
\end{lem}
\begin{proof}
For (i), note that $\kappa^* \to 0$ since the solution is assumed asymptotically complementary.  Then any point satisfying $G(v^*) = 0$ also satisfies the KKT conditions \eqref{eqn:KKT} following rescaling of $(x^*,s^*,z^*)$ by $\tau^*$.   

For (ii), suppose that $\{v^k\} \in \mathcal{C}, k = 1,2,\dots,$ is any bounded sequence with limit $v^*$.   Since $\tau^k \to 0$ by assumption, it follows that $(x^k)\tpose Px^k \to 0$ since $\frac{1}{\tau^k}x^kPx^k$ must remain bounded.   Hence $Px^* = 0$ since $P$ is assumed positive semidefinite, and both $Ax^* - s^* = 0$ and $A\tpose z^* = 0$.   Since $\{s^k,z^k\} \in \mathcal{K} \times \mathcal{K}^*$ by assumption and both cones are closed, $z^* \in \mathcal{K}^*$ and $Ax^* \in \mathcal{K}$. Since $q\tpose x^* + b\tpose z^* = -\kappa^*$ with $\kappa^* > 0$, then at least one of the inner product terms must be negative.   Consequently, if $\innerprod{q}{x^*} < 0$ then $x^* \in \mathbb{P}$ and \ref{eqn:dual} is infeasible, and if $\innerprod{b}{z^*} < 0$ then $z^* \in \mathbb{D}$ and \ref{eqn:primal} is infeasible. 
\end{proof}

\subsection{Supported constraint types and barrier functions}

Our interior-point method supports a variety of cones of both symmetric (i.e.\ self dual) and nonsymmetric type, and we allow the cone $\mathcal{K}$ in problem \ref{eqn:primal} to be any arbitrary composition of the basic cone types we describe in this section.    For the symmetric case, we support the following cones:
\begin{itemize}
	\item The \emph{nonnegative cone} $\coneNN{n} \eqdef
  \set{x \in \mathbb{R}^n}{x \ge 0}$.
	\item The \emph{second order cone} $\coneSOCP{n}  \eqdef 
  \set{(u,v) \in \mathbb{R} \times \mathbb{R}^{n-1}}{ \|v\|_2 \le u}$.
	\item The \emph{positive semidefinite cone} 
  $\coneSDP{n} \eqdef \set{x \in \mathbb{R}^{n(n+1)/2}}{\textrm{mat}(x) \succeq 0}$.
\end{itemize}
 
Our implementation for $\coneSDP{n}$ uses the standard scaled and vectorized `triangular' format.  Given some symmetric matrix $M \in \Re^{n\times n}$, we define an operator $\svec(\cdot)$ that extracts the upper triangular part and scales off-diagonal terms by $\sqrt{2}$, i.e.\  
\[
\svec(M) := 
{
(M_{1,1}, \sqrt{2}M_{1,2}, M_{2,2},\cdots,\sqrt{2}M_{n-2,n-1},M_{n-1,n-1},\sqrt{2}M_{1,n},\cdots,\sqrt{2}M_{n-1,n}, M_{nn})}.
\] 
The inverse operation $\mat(\cdot)$ then restores the original matrix, i.e.\ $M = \mat(\svec(M))$.   This scaling is necessary to preserve inner products, i.e.\ $\trace({L\tpose M}) = \innerprod{\svec(L)}{\svec(M)}$ for matrices of compatible dimension.    

We also support two nonsymmetric cones, both defined as subsets of $\mathbb{R}^3$: 
\begin{itemize}
	\item The \emph{exponential cone} $\coneEXP \eqdef 
  \set{(s_1,s_2,s_3) \in \mathbb{R}^3}{s_2 \cdot e^{s_1/s_2} \le s_3,s_2>0}$.
	\item The \emph{power cone} $\conePOW := \set{(s_1,s_2,s_3) \in \mathbb{R}^3}{s_1^{\alpha}s_2^{1-\alpha} \ge |s_3|,s_1 \ge 0,s_2 \ge 0}$.
\end{itemize}
	
For these two nonsymmetric cones the corresponding 
dual cones are
\begin{align*}
   \coneEXP^* &\eqdef \set{(z_1,z_2,z_3) \in \mathbb{R}^3}
   {z_3 \ge - z_1 \cdot e^{z_2/z_1-1}, \ z_3 > 0, z_1 < 0 }  \\
   \conePOW^* &\eqdef \set{(z_1,z_2,z_3) \in \mathbb{R}^3}{\left(\frac{z_1}{\alpha}\right)^{\alpha} \cdot \left(\frac{z_2}{1-\alpha}\right)^{1-\alpha} \ge |z_3|, \ z_1,z_2 \ge 0}.
\end{align*}

Finally, we support the zero cone $\mathcal{Z}_n := \{0\}^n$ to enable modelling of problems with equality constraints.   


\subsubsection*{Barrier functions}

For each nonzero cone we define a strictly convex \emph{barrier function} $f : \mathcal{K} \to \Re$ and its associated \emph{conjugate barrier} $f^* : \mathcal{K}^* \to \Re$.   In particular, we use barrier functions that meet the following definition:

\begin{definition}[\cite{Nesterov1994,Dahl21}]
A function $f : \mathcal{K} \to \Re$ is called a \emph{logarithmically homogeneous self-concordant barrier} with degree $\nu \ge 1$ ($\nu$--LHSCB) for the cone $\mathcal{K}$ if it satisfies the following properties:
\begin{align*}
	\begin{aligned}
		&|\nabla^3 f(x)[r,r,r]| \le 2\left(\nabla^2 f(x)[r,r]\right)^{3/2} & \forall x \in \interior \mathcal{K}, r \in \mathbb{R}^d, \\
		&f(\lambda x) = f(x) - \nu \log(\lambda) & \forall x \in \interior\mathcal{K}, \lambda > 0.
	\end{aligned}
\end{align*}
\end{definition}

The conjugate function $f^* : \mathcal{K}^* \to \Re$ is defined as
\begin{align}
	f^*(y) := \sup_{x \in \mathcal{K}} \{-\innerprod{y}{x}  - f(x)\}. \label{fundamental-conjugate-barrier-1}
\end{align}
The function $f^*$ is also $\nu$--LHSCB for $\mathcal{K}^*$ if $f$ is \cite{Nesterov97}.
The primal gradient $\gradient{f}$ satisfies
\begin{align}
	\innerprod{x}\,{\gradient{f}(x)} = - \nu, \forall x \in \interior\mathcal{K}. \label{fundamental-nu-equality}
\end{align}
In cases where the conjugate barrier function is known only through the definition \eqref{fundamental-conjugate-barrier-1} (i.e.\ rather than being representable in closed-form), we can compute its gradient as the solution to
\begin{align}
	\gradient{f^*}(y) := - \arg \sup_{x \in \interior\mathcal{K}} \{-\innerprod{y}{x} - f(x)\}. \label{fundamental-conjugate-gradient}
\end{align} 

The relations 
\eqref{fundamental-conjugate-barrier-1}--\eqref{fundamental-conjugate-gradient} then collectively ensure that
\begin{equation*} 
\begin{aligned}
	f^*(y) &= -\innerprod{y}{(-\gradient{f^*}(y))} - f(-\gradient{f^*}(y))  \\
  &= -\nu - f(-\gradient{f^*}(y)). 
\end{aligned}
\end{equation*}
and also
\begin{align*}
	-\gradient{f^*}(-\gradient{f}(x)) = x, \ \forall x \in \interior\mathcal{K}, \quad -\gradient{f}(-\gradient{f^*}(y)) = y, \ \forall y \in \interior\mathcal{K}^*. 
\end{align*}

\subsection{The central path}\label{subsection:central-path}

We assume that $f : \mathcal{K} \to \Re$ is a \vLHSCB function on $\mathcal{K}$ with conjugate barrier $f^*$ for some $\nu \ge 1$.  Given any initial $v^0 \in \mathcal{C}$, we define the \emph{central path} $v^*(\mu)$  as the unique solution to 
\begin{subequations}\label{eqn:central_path_compact}
\begin{gather}
  G(v) = \mu G(v^0)\label{eqn:central_path_compact:a} \\
  s = -\mu \nabla f^*(z), ~ z = -\mu \nabla f(s),
  \label{eqn:central_path_compact:b} 
\end{gather}
which implies that  
$$
\innerprod{s}{z} / \nu = \mu.
$$
 If the cone $\mathcal{K}$ is symmetric, then we can replace the condition \eqref{eqn:central_path_compact:b} with a simpler condition defined in terms of the \emph{Jordan product} \cite{Sturm:2002,cvxopt} on $\mathcal{K}$:
\begin{equation}\label{eqn:central_path_compact:c}
 s \circ z = \mu \mbf e,
\end{equation} 
\end{subequations}
where $\mathbf{e}$ is the standard idempotent for $\mathcal{K}$. The core of our interior-point algorithm amounts to a Newton-like method for computing a solution to the system of equations \eqref{eqn:central_path_compact}.   We describe the calculation of step directions for this method in \S\ref{sec:step_directions}.   Before doing so, we address solver initialization,  termination and preprocessing steps in the remainder of this section.


\subsection{Solver initialization}\label{sec:initialization_scaling}

Before solving we perform an equilibration step on all matrix-valued data using the Ruiz equilibration technique described in \cite{Ruiz:2001}.   We refer the reader to \cite[\S 3.5]{COSMO} and \cite[\S 5.1]{OSQP} for implementation details.    

We have several strategies for choosing an initial iterate  in the interior of our problem's conic constraints while being sufficiently near to the central path.   

\subsubsection{Symmetric cones}
In the symmetric case, i.e.\ when $\mcf K$ is self-dual, our method follows the approach of~\cite{cvxopt} for initialization and we distinguish between the cases $P = 0$ and $P\neq 0$.    Although the zero cone is not itself symmetric, we treat problems where $\mathcal{K}$ is a composition of symmetric and zero cones as symmetric for the purposes of initialization.

If $P \neq 0$, we first solve the linear system
\begin{align*}
	\left[\begin{array}{cc}
		P & A^\top\\
		A & - I
	\end{array}\right]\left[\begin{array}{c}
		x \\
		z
	\end{array}\right]=\left[\begin{array}{c}
		-q \\
		b 
	\end{array}\right],
\end{align*}
which is the solution to the problem
\begin{align*}
	\min & \quad \frac{1}{2}x^\top P x + q^\top x + \frac{1}{2}\|Ax - b\|_2^2.
\end{align*}
We set $x^0 = x$ and 
\begin{align*}
	s^0 = \left\{
	\begin{array}{ll}
		-z, & \alpha_p < -\epsilon \\
		-z + (\epsilon + \alpha_p) \mathbf{e} & \text{otherwise,}
	\end{array}
	\right.
\end{align*}
where $\alpha_p = \inf \{\alpha \ | \ -z + \alpha \mathbf{e} \in \cal K \}$. We introduce a threshold $\epsilon > 0$ to ensure that $s^0$ is in the strict interior of the cone $\mathcal K$. Likewise, $z^0$ is set to
\begin{align*}
	z^0 = \left\{
	\begin{array}{ll}
		z, & \alpha_d < -\epsilon \\
		z + (\epsilon + \alpha_d) \mathbf{e} & \text{otherwise,}
	\end{array}
	\right.
\end{align*}
where $\alpha_d = \inf \{\alpha \ | \ z + \alpha \mathbf{e} \in \cal K \}$. 

If $P = 0$, we instead solve the linear system
\begin{align*}
	\left[\begin{array}{cc}
		0 & A^\top\\
		A & - I
	\end{array}\right]\left[\begin{array}{c}
		x \\
		s
	\end{array}\right]=\left[\begin{array}{c}
		0 \\
		b 
	\end{array}\right],
\end{align*}
in an effort to minimize the linear residuals.
We set $x^0 = x$ and 
\begin{align*}
	s^0 = \left\{
	\begin{array}{ll}
		-s, & \alpha_p < -\epsilon \\
		-s + (\epsilon + \alpha_p) \mathbf{e} & \text{otherwise},
	\end{array}
	\right.
\end{align*}
where $\alpha_p = \inf \{\alpha \ | \ -s + \alpha \mathbf{e} \in \cal K \}$. We then solve the system
\begin{align*}
	\left[\begin{array}{cc}
		0 & A^\top\\
		A & - I
	\end{array}\right]\left[\begin{array}{c}
		x \\
		z
	\end{array}\right]=\left[\begin{array}{c}
		-q \\
		0 
	\end{array}\right],
\end{align*}
which is equivalent to
\begin{align*}
\begin{aligned}
		\min_z & \quad \frac{1}{2}\|z\|_2^2 \\
	s.t. & \ A^\top z + q = 0.
\end{aligned}
\end{align*}
Finally, we set
\begin{align*}
	z^0 = \left\{
	\begin{array}{ll}
		z, & \alpha_d < -\epsilon \\
		z + (\epsilon + \alpha_d) \mathbf{e} & \text{otherwise,}
	\end{array}
	\right.
\end{align*}
where $\alpha_d = \inf \{\alpha \ | \ z + \alpha \mathbf{e} \in \cal K \}$.    In either case, we set the scalars $(\tau^0, \kappa^0)$ to $1$.

\subsubsection{Nonsymmetric cones}
When $\mcf K$ contains any nonsymmetric cone, we instead apply a unit initialization strategy as in~\cite{Skajaa15,Dahl21}. In this case, we initialize both primal and dual variables at a point on the central path satisfying \mbox{$z = s = - \nabla f^*(z)$} (corresponding to $\mu^0 = 1$).   This is equivalent to solving the unconstrained optimization
$$
\min_{z} \ \frac{1}{2}z^2 + f^*(z),
$$
which is strictly convex and has a unique solution.
It yields $s_{sym}^0 = z_{sym}^0= \mathbf{e}$ for symmetric cones,
\begin{align*}
  s_{\text{exp}}^0 &= z_{\text{exp}}^0 \approx (-1.051383,0.556409,1.258967)
\intertext{
for exponential cones, and
}
	s_{\text{pow}}^0 &=  z_{\text{pow}}^0= (\sqrt{1+\alpha}, \sqrt{2-\alpha},0)
\end{align*}
for power cones of parameter $\alpha$.

\subsection{Chordal Decomposition}\label{sec:chordal_decomposition}

For semidefinite programs (SDPs) we follow the chordal-decomposition and clique merging strategy developed as part of the COSMO solver \cite{COSMO}.   We implement both a `parent-child' merge strategy based on the clique-tree analysis method of \cite{Sun:2014} and the clique-graph merge strategy described in \cite[\S 5]{COSMO}.  We reformulate problems decomposed using either of these methods using the so-called `compact' or `range-space' conversion method described in \cite[\S 5]{Kim:2011}.

When implementing the clique-graph based merge strategy, we use the same edge-weight metric as in~\cite{COSMO} to determine when candidate clique merges are accepted.   Given a clique graph with $V$ vertices and two cliques $\mathcal{C}_i$ and $\mathcal{C}_j$, we define an edge weight function $e : 2^V \times 2^V \to \Re$ as
\[
  e(\mathcal{C}_i,\mathcal{C}_j) = |\mathcal{C}_i|^3  + |\mathcal{C}_i|^3 - |\mathcal{C}_i \cup \mathcal{C}_j|^3,
\]
and merge candidate cliques when this metric is positive.  We note that this metric was used in the ADMM-based solver \cite[\S 5]{COSMO} because it relates directly to the computational complexity of projections onto the positive semidefinite cone.   In our case the size of a Hessian block generated by a conic variable in $\coneSDP{n}$ has dimension $\Re^{(n(n+1)/2) \times (n(n+1)/2)}$, with the resulting KKT factorization time further complicated by the linking variables generated between overlapping cliques.   We note therefore that it is likely that more efficient merge metrics are possible.

\subsection{Termination Criteria}

All termination criteria are based on {unscaled} problem data and iterates, i.e.\ after the Ruiz scaling described in \S\ref{sec:initialization_scaling} has been reverted.    For checks of primal and dual feasibility, we first define normalized variables $\bar{x} = x/\tau, \bar{s} = s/\tau, \bar{z} = z/\tau$ and then define primal and dual residuals as 
\begin{align*}
r_p &\eqdef A\bar x - \bar s + b \\
r_d &\eqdef P\bar x + A^\top \bar z + q 
\intertext{and primal and dual objectives as}
  g_p &\eqdef \frac{1}{2} \bar x^\top P \bar x + q \tpose x  \\
  g_d &\eqdef -\frac{1}{2} \bar x^\top P \bar x - b\tpose \bar z .
\end{align*}

We then declare convergence if each of the following three conditions holds:
\begin{align*}
  \norm{r_p} &< \epsilon_f \cdot \max\{1, \norm{b}_\infty + \norm{\bar x} + \norm{\bar s}\} \\
  \norm{r_d} &< \epsilon_f \cdot \max\{1, \norm{q}_\infty + \norm{\bar x} + \norm{\bar z}\} \\
  |g_p - g_d| &< \epsilon_f\cdot  \max\{1, \min\{|g_p|, |g_d|\}\}.
\end{align*}

We specify a default value of $\epsilon_f = 10^{-8}$ in our implementation, and a weaker threshold of $\epsilon = 10^{-5}$ when testing for `near optimality' in cases of early termination (e.g.\ lack of progress, timeout, iterations limit).

When testing for infeasibility certificates, we do \emph{not} normalize iterates, but rather work directly with the unscaled variables since infeasibility corresponds to the case where $\tau \to 0$.  We declare primal infeasibility if 
\begin{align*}
\norm{A\tpose z} & < -\epsilon_{i,r} \cdot \max(1, \norm{x} + \norm{z}) \cdot (b\tpose z) \\
b\tpose z & < -\epsilon_{i,a},
\end{align*}
and dual infeasibility if 
\begin{align*}
  \norm{Px} & < -\epsilon_{i,r} \cdot \max(1, \norm{x}) \cdot (b\tpose z) \\
  \norm{Ax + s} & < -\epsilon_{i,r} \cdot \max(1, \norm{x} + \norm{s}) \cdot (q\tpose x) \\
  q\tpose x  & < -\epsilon_{i,a}.
  \end{align*}
We again set $\epsilon_{i,r} = \epsilon_{i,a} = 10^{-8}$ as default values, and allow for weaker thresholds to declare `near infeasibility' certificates in cases of early termination.

\section{Computing step directions}\label{sec:step_directions}

Our interior-point method computes Newton-like search directions using a linearization of \eqref{eqn:central_path_compact} given some right-hand side residual $d := (d_x, d_z,d_\tau,d_s,d_\kappa)$.  This produces a linear system in the form

\begin{subequations}\label{eqn:linsys_4x4}
\begin{gather}%
\begin{bmatrix}
  0 \\ \Delta s \\ \Delta \kappa
\end{bmatrix}  
-
\begin{bmatrix}
  P & A\tpose & q \\
  -A  & 0 & b \\
  -(q + 2P\xi)\tpose & -b\tpose & \xi\tpose P \xi
\end{bmatrix}
\begin{bmatrix}
  \Delta x \\ \Delta z \\ \Delta \tau
\end{bmatrix}  
= 
-
\begin{bmatrix}
  d_x\\d_z \\ d_\tau
\end{bmatrix}  \label{eqn:linsys_4x4:a} \\[1ex] 
H \Delta z + \Delta s = -d_s, \quad \kappa \Delta \tau + \tau\Delta \kappa = - d_\kappa,
\label{eqn:linsys_4x4:b}
\end{gather}
\end{subequations}
where $\xi := x\tau^{-1}$ and $H \in \Re^{m\times m}$ is a positive definite matrix that we will refer to as the \emph{scaling matrix}.   The choice of both the scaling matrix and the right-hand side terms $d$ in \eqref{eqn:linsys_4x4}  depend on the search direction, the symmetry or asymmetry of the cone $\mcf{K}$ and the particular choice of scaling strategy.  We defer more precise definitions of these terms to later in this section.

Our approach to solving~\eqref{eqn:linsys_4x4} follows that of \cite{cvxopt,ecos}, but differs in the fact that some blocks in the coefficient matrix of \eqref{eqn:linsys_4x4:a} include an additional term when $P \neq 0$.   We first eliminate the variables $(\Delta s, \Delta \kappa)$ to obtain the reduced system 

\begin{subequations}\label{eqn:linsys_3x3}
	\begin{gather}
		\begin{bmatrix}
			P & A^\top & q \\
			-A & H & b \\
			-(q + {2P \xi})^\top & -b^\top & {\xi^\top P \xi} + \frac{\kappa}{\tau}
		\end{bmatrix}
    \begin{bmatrix}
			\Delta x \\
			\Delta z \\
			\Delta \tau
		\end{bmatrix} = 
    \begin{bmatrix}
			d_x \\
			d_z - d_s  \\
			d_\tau - d_\kappa / \tau
  \end{bmatrix} \\[1ex]
		\Delta s = -d_s - H \Delta z, \quad \Delta \kappa=-(d_\kappa+\kappa \Delta \tau) /\tau. \label{eqn:linsys_3x3:b}
	\end{gather}
\end{subequations}

To solve \eqref{eqn:linsys_3x3} we solve a pair of linear systems with a common left-hand side
	\begin{equation}\label{eqn:linsys_2x2}
		\begin{bmatrix}
			P & A^\top\\
			A & - H
		\end{bmatrix}
    \left[
      \begin{array}{c|c}
        \Delta x_1 & \Delta x_2\\
        \Delta z_1 & \Delta z_2
    \end{array}
    \right]
    =
    \left[
    \begin{array}{c|c}
			d_x  & -q \\
			-(d_z - d_s) & b
    \end{array}
    \right]
  \end{equation}
and then recover $(\Delta \tau, \Delta x, \Delta z)$ using 
\begin{align*}
	\Delta \tau &=\frac{d_\tau-d_\kappa /\tau+(2P\xi + q)^\top \Delta x_1+b^\top \Delta z_1}{\kappa/\tau + \xi^\top P \xi-(2P\xi + q)^\top \Delta x_2-b^\top \Delta z_2} \notag \\[1ex]
	&=\frac{d_\tau-d_\kappa /\tau + q^\top \Delta x_1+b^\top \Delta z_1 + 2\xi^\top P \Delta x_1}{\|\Delta x_2 - \xi\|_P^2 -\|\Delta x_2\|_P^2 -q^\top \Delta x_2-b^\top \Delta z_2},
\end{align*}
and
\begin{align*}
	\Delta x=\Delta x_1+\Delta \tau \cdot \Delta x_2, \quad \Delta z=\Delta z_1+\Delta \tau \cdot \Delta  z_2.
\end{align*}
Finally, we recover $(\Delta s, \Delta \kappa)$ from~\eqref{eqn:linsys_3x3:b}.
After obtaining the search direction $(\Delta s, \Delta z,\Delta \tau, \Delta \kappa)$, we  compute the maximal step size $\alpha$ ensuring that the new update resides in the interior of conic constraints. In cases where any part of $\mathcal{K}$ is nonsymmetric, we also choose $\alpha$ sufficiently small so that the updated values $(s+\alpha\Delta s, z + \alpha\Delta z)$ remain within a neighborhood of the central path~\eqref{eqn:central_path_compact} using the proximity metric described in~\cite{Dahl21}.

\subsection{Scaling matrices}
The choice of scaling matrix $H$ in~\eqref{eqn:linsys_4x4} depends on the way in which we linearize the central path equations as defined in Section~\ref{subsection:central-path}. For symmetric cones the most common choice is the NT scaling~\cite{Nesterov97,Nesterov98}, which we also employ. 
For nonsymmetric cones, the central path is defined by the set of point satisfying~\eqref{eqn:central_path_compact:b}, and both symmetric scaling~\cite{Dahl21} or the nonsymmetric scaling strategies~\cite{Serrano15} have been implemented in  state-of-the-art solvers.   We set $H = 0$ for the zero cone.

\subsubsection{Symmetric cones}
For symmetric cones we linearize the central path equation \eqref{eqn:central_path_compact:c}.
The NT scaling method exploits the self-scaled property of symmetric cone $\mathcal{K}$ to define, for $(s,z) \in \mathcal{K}$, a unique scaling point $w \in \mathcal{K}$ satisfying
\begin{equation*}
	H(w)s = z.
\end{equation*}
The matrix $H(w)$ can be factorized as $H^{-1}(w) = W^\top W$, and we set $H = H^{-1}(w)$ in~\eqref{eqn:linsys_4x4}. The factors $w,W$ are then computed following~\cite{cvxopt} except in the case of second-order cones, where we instead apply a modified version of the sparse factorization strategy of \cite{ecos} for second-order cones of dimension greater than $4$.
 
\subsubsection{Nonsymmetric cones}
In the nonsymmetric case the self-scaled property does not hold and the central path can't be formulated as in~\eqref{eqn:central_path_compact:c} due to the lack of a Jordan algebra, so we instead linearize \eqref{eqn:central_path_compact:b}. A general primal-dual scaling strategy suitable for nonsymmetric cones was consequently proposed in~\cite{Tuncel01}, and used later in~\cite{Dahl21}, 
%
%
which relies on the satisfaction of two secant equations 
\begin{align*}
	Hz = s, \quad H \nabla f(s) = \nabla f^*(z).
\end{align*}
Suppose we define \emph{shadow iterates} as
\begin{align*}
	\tilde{z} := -\nabla f(s), \quad \tilde{s}:= - \nabla f^*(z),
\end{align*}
with
\begin{align*}
	\tilde{\mu} = \langle \tilde{s}, \tilde{z} \rangle/\nu.
\end{align*}
A scaling matrix $ H $ can then be obtained from the rank-4 Broyden-Fletcher-Goldfarb-Shanno (BFGS) update, which is commonly used in quasi-Newton methods,
\begin{align*}
	H \eqdef H_{\mathrm{BFGS}} := Z(Z^\top S)^{-1}Z^\top + H_a - H_aS(S^\top H_aS)^{-1}S^\top H_a,
\end{align*}
where $Z \eqdef [z, \tilde{z}], S \eqdef [s, \tilde{s}], \tilde{z} \eqdef -\nabla f(s), \tilde{s} \eqdef - \nabla f^*(z)$ and $H_a \succ 0$ is an approximation of the Hessian. In our implementation, we choose $H_a = \mu \nabla^2 f^*(z)$ following~\cite{Dahl21} and the calculation of $H_{\mathrm{BFGS}}$ reduces to a rank-3 update,
\begin{align*}
	\begin{aligned}
		H_{\mathrm{BFGS}}=& \mu \nabla^2 f^*(z)+\frac{1}{2 \mu \nu} \delta_s\left(s+\mu \tilde{s}+\frac{1}{\mu \tilde{\mu}-1} \delta_s\right)^T+\frac{1}{2 \mu \nu}\left(s+\mu \tilde{s}+\frac{1}{\mu \tilde{\mu}-1} \delta_s\right) \delta_s^T \\
		&-\mu \frac{\left(\nabla^2 f^*(z) \tilde{z}-\tilde{\mu} \tilde{s}\right)\left(\nabla^2 f^*(z) \tilde{z}-\tilde{\mu} \tilde{s}\right)^T}{\left\langle\tilde{z}, \nabla^2 f^*(z) \tilde{z}\right\rangle-\nu \tilde{\mu}^2} . 
	\end{aligned}
\end{align*}
More details about the  primal-dual scaling of nonsymmetric cones can be found in~\cite{Dahl21}.

\subsection{The affine and centering directions}


At every interior-point iteration, our method requires us to solve the linear system~\eqref{eqn:linsys_4x4} for each of two different right-hand side terms $(d_x,d_z,d_{\tau}, d_s, d_{\kappa})$.   This amounts to a single factorization of the condensed system~\eqref{eqn:linsys_2x2} followed by three linear solves, noting the common right-hand-side term $(-q,b)$ in~\eqref{eqn:linsys_2x2}. 

The first of these two choices produces the \textit{affine step} direction, where we compute the search direction $(\Delta x, \Delta s, \Delta z, \Delta \kappa, \Delta \tau)$ to eliminate the linear residuals in \eqref{eqn:root_finding_G}, i.e.\ by computing a pure Newton step direction for \eqref{eqn:central_path_compact} with $\mu = 0$. The second is called the \textit{combined step}, which is the affine step plus a centering correction toward the central path. In practice, the estimated direction $(\Delta x, \Delta s, \Delta z, \Delta \kappa, \Delta \tau)$ from the affine step is used to determine the right-hand side $(d_x,d_z,d_{\tau}, d_s, d_{\kappa})$ for the combined step to compute a higher-order correction for acceleration. In the language of predictor-corrector algorithms, the affine step is the \emph{predictor} while the centering step as the \emph{corrector}.

Computation of a higher-order correction term $\eta$ is a heuristic technique that is known to accelerate the convergence of IPMs significantly.  The choice of this term varies depending on the choice of the scaling matrix $H$ and whether a given cone constraint is symmetric. 
For our method $(d_x,d_z,d_{\tau}, d_s, d_{\kappa})$ is defined as 
\begin{align*}
	d_x = r_x, d_z = r_z, d_{\tau} = r_{\tau}, 	d_{\kappa} = \kappa \tau, d_s = s, 
\end{align*}
in the affine step, and 
\begin{align*}
	\begin{aligned}
		d_x = (1-\sigma)r_x, d_z = (1-\sigma)r_z, d_{\tau} = (1-\sigma)r_{\tau}, d_{\kappa} = \kappa \tau + \Delta \kappa \Delta \tau - \sigma \mu, \\
		d_s = \left\{
		\begin{array}{cl}
			W^\top \left(\lambda \backslash \left(\lambda \circ \lambda + \eta - \sigma \mu \mathbf{e}\right) \right) & \text{(symmetric)}\\
			s + \sigma \mu \nabla f^*(z) + \eta & \text{(nonsymmetric),}
		\end{array}
		\right.
	\end{aligned}
\end{align*}
in the combined step.  For symmetric cones we use the Mehrotra correction $\eta = (W^{-1} \Delta s) \circ (W \Delta z)$ \cite{cvxopt}, while for nonsymmetric cones we compute $\eta$ using the 3rd-order correction method from~\cite{Dahl21}, i.e.
\begin{equation*}
	\eta = -\frac{1}{2} \nabla^3 f^*(z)[\Delta z,\nabla^2 f^*(z)^{-1} \Delta s].
\end{equation*}
We choose 
$
	\sigma = (1-\alpha_{\textrm{aff}})^3,
$
where $\alpha_{\textrm{aff}}$ is the step size from the affine step.

\subsection{Linear solve method} \label{sec:linearsolvemethods}

Nearly all of the cost of computing step directions in \S\ref{sec:step_directions} typically comes when solving the symmetric linear system \eqref{eqn:linsys_2x2}.   To solve this system we add a small regularization term $\epsilon_{s}$ and then compute a direct factorization
\[
K :=
\begin{bmatrix}
  P + \epsilon_sI & A^\top\\
  A & - (H+\epsilon_sI)
\end{bmatrix} = LDL\tpose.
\]
This \emph{static regularization} ensures that the matrix $K$ is quasidefinite \cite{Vanderbei1995} even if $P$ is rank deficient, and consequently that the factor $D$ is diagonal with $D_{ii} \neq 0$ when computing the Cholesky-like factorization $LDL^T$~\cite{Gill1996}.   This approach also ensures that the sparsity pattern of the factor $L$ depends only on the problem's sparsity pattern, so the method is allocation-free after the first factorization.  When computing the $LDL$ factorization we also employ a \emph{dynamic regularization} strategy by bounding pivots away from zero by an amount $\epsilon_d$ to ensure that the factorization is numerically stable.   We have extended the open-source package QDLDL, original developed for OSQP and based on \cite{Davis2005}, to support this regularization strategy.   

  The improved numerical efficiency of our scheme relative to the standard HSDE method can be seen by consideration of the linear system in \eqref{eqn:linsys_2x2}.   If we had started instead from the epigraphical reformulation~\eqref{eqn:primal_linear}, then the coefficient matrix in \eqref{eqn:linsys_2x2} would have been 
  \begin{equation}\label{eqn:hsde_bad_matrix_form}
    \begin{bmatrix}
			0 & A^{T} & [2P^{\frac{1}{2}}]^T\,\\
			A & -H^{\hphantom{T}} \\
      2P^{\frac{1}{2}} & & -H_\mathcal{K}
		\end{bmatrix}
  \end{equation}
  with $H_\mathcal{K}$ a scaling matrix appropriate for the additional second order cone constraint introduced in~\eqref{eqn:primal_linear}.   Direct factorization of \eqref{eqn:hsde_bad_matrix_form} can produce significantly more fill-in than for coefficient matrix \eqref{eqn:linsys_2x2}, particular when the matrix factor $P^\frac{1}{2}$ already has substantial fill-in.

\section{Implementation: the Clarabel solver}

We have implemented our approach in the \textsc{Clarabel} solver, an open-source and liberally licensed software package with separate implementations in both the Rust \cite{rust} and Julia \cite{julia} programming languages.   Both implementations, along with extensive documentation, are available under the Apache v2.0 license and are publicly available via our main project site at 

\begin{center}
\color{blue}\url{https://clarabel.org/}
\end{center}

Our Rust implementation is intended for most academic and industrial end users.  This implementation can also be accessed via other common languages through standard foreign function interfaces (FFIs) and wrappers, and we currently provide such wrappers for Python, C, \Cpp and R.   We also provide interfaces in Python to the standard modelling package CVXPY \cite{diamond2016cvxpy,agrawal2018rewriting}.   Clarabel is installed as part of the standard CVXPY distribution \cite{cvxpy:web} and is the default solver in CVXPY for linear and second-order cone programs as of version 1.5.

The Rust version of Clarabel provides its own internal implementation of most linear algebra functionality, including a standalone Rust reimplementation of the quasidefinite linear solver QDLDL with regularization features as described in \S\ref{sec:linearsolvemethods}.   We use Rust interfaces to user-selectable BLAS~\cite{blackford2002updated} implementations (e.g. \cite{intel_mkl,openblas}) for solving conic programs on the semidefinite cone.   We also allow for the use of alternative direct linear solvers, and provide optional support for a 3rd-party multithreaded supernodal LDL factorization method implemented in native Rust as part of the \texttt{faer-rs} package \cite{faer}.

The Julia implementation is intended for use both as a standalone solver for users of the Julia language and as a prototyping platform for future algorithmic development.  The Julia implementation relies heavily on native Julia functions for most linear algebra functionality, aside from a Julia implementation of QDLDL which we provide as a standalone package.  In Julia, we also provide the option of using alternative linear solve methods including CHOLMOD \cite{cholmod}, Pardiso \cite{pardiso} and the HSL MA57 \cite{ma57} solver. 

Both implementations provide the same functionality and support the same set of conic constraints.   We also provide support for different floating point data types in both languages, e.g.\ for standard 32- or 64-bit single or double precision floating point types \cite{kahan1996ieee} or for extended precision types such as the Julia \textsc{BigFloat} type.   Our implementation is inspired by the modular design pattern of the interior-point solver OOQP~\cite{ooqp}, in the sense that all internal data types are defined as abstract types that can be extended or customized by end users to specific problem classes to exploit domain-specific structure.  In the Rust implementation this functionality relies heavily on Rust's trait-based type system and generics, while in Julia we instead rely on Julia's dynamic dispatch and ``duck typing'' \cite{ducktyping}. 

For more detail we refer the reader to the documentation available on the \textsc{Clarabel} project website.

\section{Numerical Experiments}

We have benchmarked our implementation of Clarabel against a variety of open-source and commercial solvers: the open-source interior-point solver ECOS \cite{ecos}, the open-source dual-simplex base solver HiGHS \cite{HiGHS}, and the commercial interior-point solvers Mosek~\cite{MOSEK} and Gurobi~\cite{gurobi}.   All benchmarks are executed with all default settings for all solvers enabled, but with pre-solve disabled where applicable to ensure that the solvers are solving equivalent problems.   We do not impose any iteration limits other than those specified within each solver's internal defaults.  We set the maximum solve time to 300 seconds unless otherwise stated.

Our two implementations of Clarabel use our native implementations of the direct LDL linear solver QDLDL for all tests unless otherwise stated.   The  QDLDL solver is relatively unsophisticated relative to the multithreaded methods used in commercial solvers, but is lightweight, open-source and does not rely on any external libraries. 

All experiments were carried out on the Oxford University Advanced Research Computing (ARC) Facility~\cite{OxfordArc}.  For each test problem in our benchmarks results, solver were run single threaded on an Intel Xeon Platinum 8268 CPU @ 2.90 GHz with 64 GB RAM.  All benchmarks tests are scripted in Julia and access solver interfaces via JuMP \cite{Lubin2023}.  We use Rust compiler version 1.72.0 and Julia version 1.9.2.   The code for all numerical examples is publicly available~\cite{ClarabelBenchmarks}.

For each set of benchmark problems in our results we exclude problems for which \emph{none} of the benchmarked solvers produced a valid solution.   We provide a summary of the results for all benchmarked solvers appropriate to each problem class in the form of shifted geometric means and performance profiles in the remainder of this section.

For all benchmark tests sets, we provide more detailed numerical results for our Rust and Julia implementations as well as the solvers Mosek and ECOS (where applicable) in Appendix~\ref{appendix:full_results}, including solve times and iteration counts.   We include only this subset of solvers in our detailed reporting since all are interior-point based methods and therefore have iteration counts that are directly comparable.  

\clearpage

\paragraph{Shifted geometric means}

We follow the standard benchmarking convention of using a normalized shifted geometric mean for comparison of solve time across different solvers \cite{hansbench}.   For a set of $N$ test problems, we define the shifted geometric mean solve time $g_s$ for solver $s$ as 
\[
g_s \eqdef \left[\prod_{p=1}^N  (t_{p,s} + k) \right]^{\frac{1}{N}} - k,
\]

where $t_{p,s}$ is the time in seconds for solver $s$ to solve problem $p$, and $k = 1$ is the shift.  The normalized shifted geometric mean is then defined as 
\[
  r_s \eqdef \frac{g_s}{\min_{s} g_{s}},
\]
so that the solver with the lowest shifted geometric mean solve time has a normalized score of~1.  For those problems for which a given solver fails, we assign a solve time $t_{p,s}$ equal to the maximum allowable solve time for the relevant benchmark.

\paragraph{Performance profiles}

We also provide performance profiles \cite{Dolan2002} to compare both the relative and absolute performance of different solvers.   For a set of $N$ test problems, we define the \emph{ relative performance ratio} for solver $s$ and problem $p$ as
\[
  u_{p,s} = \frac{t_{p,s}}{\min_{s} t_{p,s}}.
\]
The performance profile for the solver $s$ is then a plot of the function $f^r_s : \Re_+ \mapsto [0,1]$ defined as 
\[
f^r_s(\tau) \eqdef \frac{1}{N} \sum_p \mathcal{I}_{\le \tau}(u_{p,s}),
\]
where $\mathcal{I}_{\le \tau} = 1$ if $\tau \le t_{p,s}$ and $\mathcal{I}_{\le \tau} = 0$ otherwise.   The relative performance profile therefore shows, at each level $\tau$, the fraction of problems solved by solver $s$ in time within a factor $\tau$ of the solve time of the best solver. 

Since the relative performance profile for a given solver can change depending on the overall collection of solvers being benchmarked, we further compute an \emph{absolute performance profile} by plotting a function $f^a_s : \Re_+ \mapsto [0,1]$ as 
\[
  f^a_s(\tau) \eqdef \frac{1}{N} \sum_p \mathcal{I}_{\le \tau}(t_{p,s}).
\]
The absolute performance profile then shows, at each level $\tau$, the fraction of problems solved by solver $s$ within $\tau$ seconds and is independent of the other solvers benchmarked.

\subsection{Benchmark problems with quadratic objectives}

In this section we present benchmark results for quadratic programming (QP) problems in the standard form \ref{eqn:primal} with the set $\mathcal{K}$ restricted to the composition of the zero cone (i.e.\ modelling linear equality constraints) and the nonnegative orthant.  We consider example problems taken or generated from standard open-source problem collections and covering a wide range of problem dimensions. 

\subsubsection{The Maros-Meszaros test set}

We consider first the standard benchmark collection of 138 quadratic programs from the Maros-Meszaros test set \cite{maros1999}.   This collection of QPs includes a wide range of problem sizes and contains a number of difficult test cases due to numerical ill-conditioning, rank deficiency or poor scaling. 

\textbf{Results} Results for this benchmark set are shown in Figure~\ref{fig:maros} for all solvers.   Clarabel is the fastest overall solver on this benchmark set, with the Rust implementation marginally faster as expected.   We observe that the seemingly large gap in the relative performance profile of our Rust and Julia implementations is almost entirely attributable to faster solve times among the smallest examples in this test set.   All solvers fail on at least some subset of these benchmarks, with Gurobi the lowest failure rate at full accuracy, and Clarabel the lowest failure rate at reduced accuracy.

Of particular note in this benchmark set is the high failure rate of the ECOS and Mosek solvers, since both are interior-point methods broadly similar to Clarabel.   In the case of ECOS these failures are partly attributable to the solvers' requirement to reformulate QP problems in the conic form \eqref{eqn:hsde_bad_matrix_form}, since ECOS does not support quadratic objectives natively.   This leads to immediate failures in a substantial number of cases due to ill-conditioning of the matrix $P$, resulting in a failed attempt to compute the Cholesky factor $P^{\frac{1}{2}}$ in \eqref{eqn:hsde_bad_matrix_form} when $P$ is either semidefinite or contains very small negative eigenvalues.  Mosek handles this case more robustly, but is still not able to solve a substantial number of problems to full accuracy within the benchmark time limit.   

\captionsetup{labelfont=bf}
\begin{figure}[ht]
  \centering
  \caption{\bf Performance profiles for the Maros-Meszaros problem set}
  \label{fig:maros}
  \begin{subfigure}[b]{0.49\textwidth}
      \centering
      {\includegraphics[width=\textwidth]{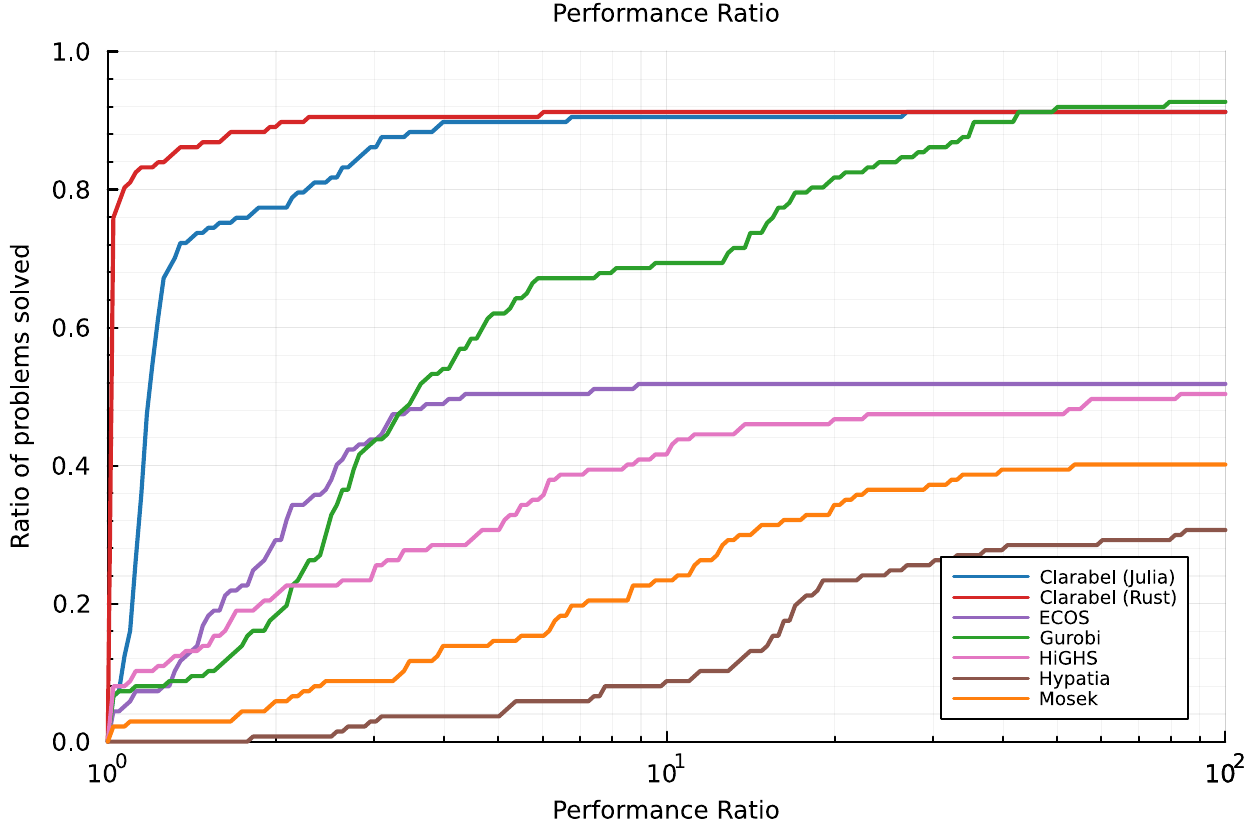}}
      \caption{Relative performance profile}
      \label{fig:maros:relative}
  \end{subfigure}
  \hfill
  \begin{subfigure}[b]{0.49\textwidth}
      \centering
      {\includegraphics[width=\textwidth]{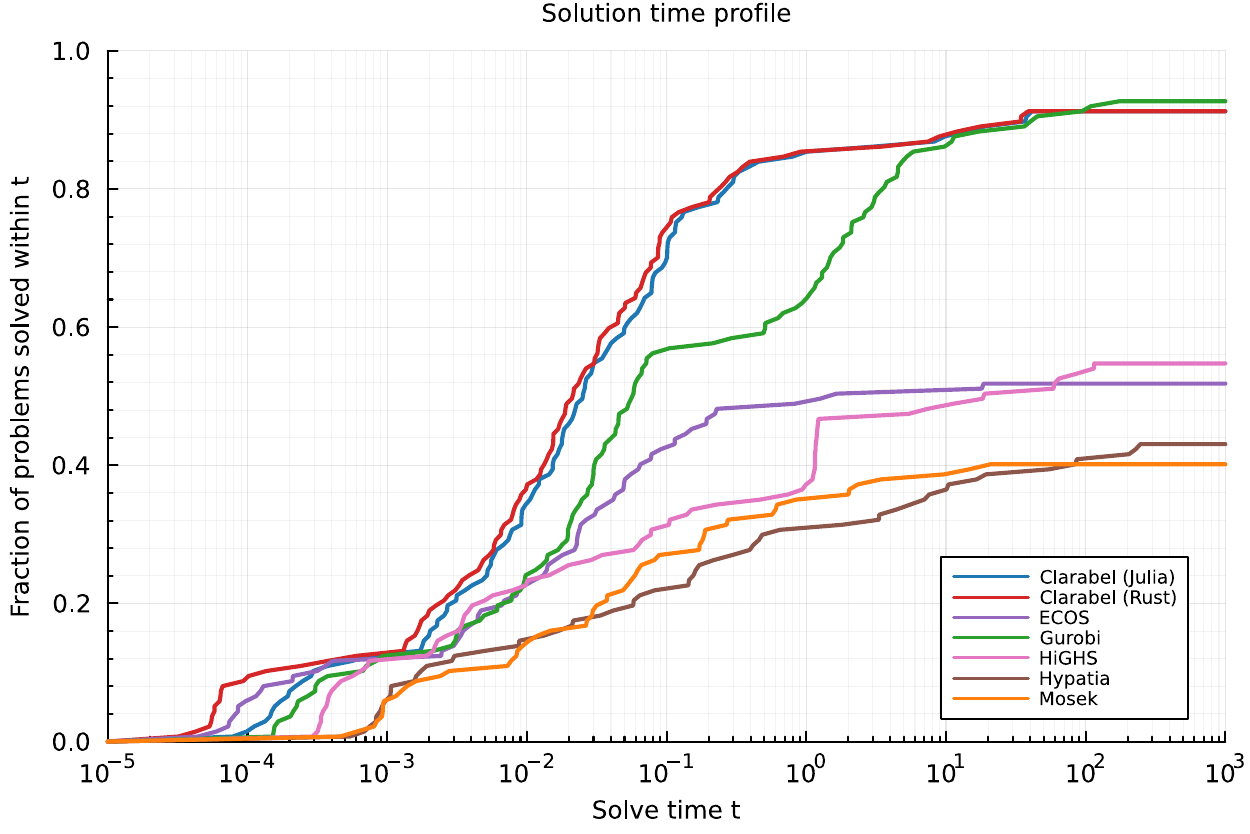}}
      \caption{Absolute performance profile}
      \label{fig:maros:absolute}
  \end{subfigure}
  \begin{subfigure}{1\textwidth}
    \centering
    \footnotesize
    \begin{tabular}{llccccccc}
  \hline
   &  & \textbf{ClarabelRs} & \textbf{Clarabel} & \textbf{ECOS} & \textbf{Gurobi} & \textbf{HiGHS} & \textbf{Hypatia} & \textbf{Mosek} \\\hline
  Shifted GM & Full Acc. & 1.0 & 1.02 & 15.49 & 1.64 & 17.92 & 36.61 & 32.67 \\
   & Low Acc. & 1.0 & 1.1 & 19.8 & 2.9 & 39.99 & 42.99 & 4.07 \\
  Failure Rate (\%) & Full Acc. & 8.8 & 8.8 & 48.2 & 7.3 & 45.3 & 56.9 & 59.9 \\
   & Low Acc. & 2.2 & 2.9 & 38.0 & 3.6 & 45.3 & 42.3 & 10.2 \\\hline
\end{tabular}

    \caption{Benchmark timings as shifted geometric mean and failure rates}
  \end{subfigure}
\end{figure}


\subsubsection{Least-squares problems with SuiteSparse matrices}

We next consider a collection of 23 sparse least-squares problems $Ax\approx b$ derived from matrices taken from the SuiteSparse Matrix Collection \cite{SuiteSparseMatrices}, following the equivalent set of benchmark examples from \cite{OSQP}.  For each case we compute an approximate solution by solving a constrained QP using two standard methods:

\paragraph{Huber Problem:}
The \emph{Huber fitting} \cite{Huber:1964,Huber:1981} or \emph{robust least squares} problem for a given matrix $A$ and vector $b$ is defined as 
\[
\min_{x}\,\,\,{\displaystyle\sum\limits_{i = 1}^m\phi(a_i^T x - b_i),}
\]
where $a_i\tpose$ is the $i^\textrm{th}$ row of $A$ and the \emph{Huber loss function} $\phi : \Re \to \Re$ is defined as
\begin{equation*}
\phi(w) = \begin{cases}
w^2 & |w| \le M \\
M(2|w| - M) & |w| > M.
\end{cases}
\end{equation*}

We set $M = 1$ for all test cases.   This problem is equivalent \cite{Mangasarian:2000} to the following quadratic program:
\MinProblem{x,u,r,s}{u\tpose u + 2M{\bf{1}}\tpose (r+s)}
{
  \begin{aligned}[t] \nonumber
  Ax - b - u &= r-s \\
  (r,s) &\ge 0.
  \end{aligned}
}

\paragraph{LASSO Problem:}
The \emph{least absolute shrinkage and selection operator (LASSO)} problem \cite{Tibshirani:1996,Candes:2008} for a given matrix $A$ and vector $b$ is defined as
\UnconstrainedMinProblem{x}{\norm{Ax-b}^2_2 + \lambda \norm{x}_1.\nonumber}
We set $\lambda = \norm{A\tpose b}_\infty$ for all test cases.  This problem is equivalent \cite{Mangasarian:2000} to the following quadratic program:
\MinProblem{y,x,t}{y\tpose y + \lambda {\bf{1}}\tpose t}
{
  \begin{aligned}[t] \nonumber
  &Ax - b = y \\
  &-t \le x \le t.
  \end{aligned}
}

\paragraph{Results} Results for this benchmark set of 46 problems are shown in Figure~\ref{fig:sslsq}.   Clarabel is again the fastest solver overall with the Rust implementation marginally faster.  In this test set only the Clarabel and Gurobi solvers are able to solve all cases to full accuracy.   

\captionsetup{labelfont=bf}
\begin{figure}[t]
  \centering
  \caption{\bf Performance profiles for the SuiteSparse least-squares problem set}
  \label{fig:sslsq}
  \begin{subfigure}[b]{0.49\textwidth}
      \centering
      {\includegraphics[width=\textwidth]{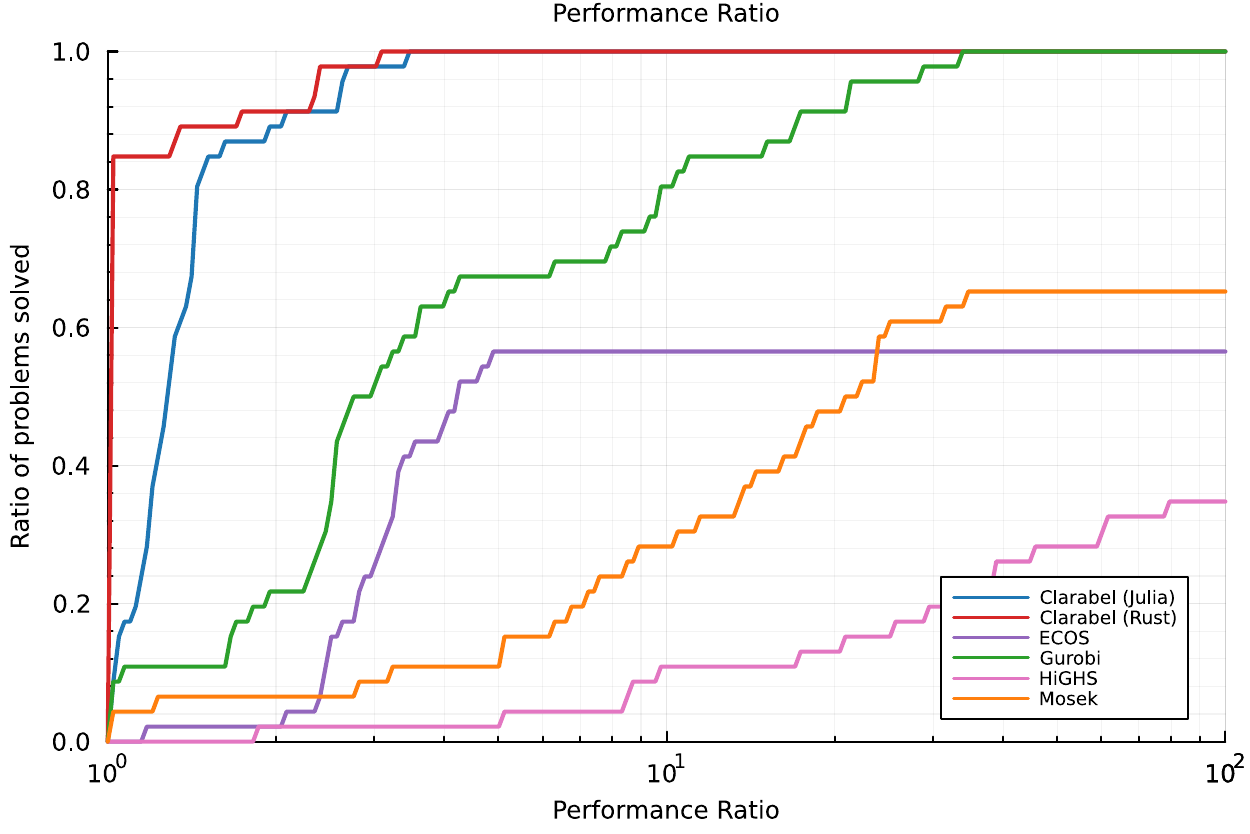}}
      \caption{Relative performance profile}
      \label{fig:sslsq:relative}
  \end{subfigure}
  \hfill
  \begin{subfigure}[b]{0.49\textwidth}
      \centering
      {\includegraphics[width=\textwidth]{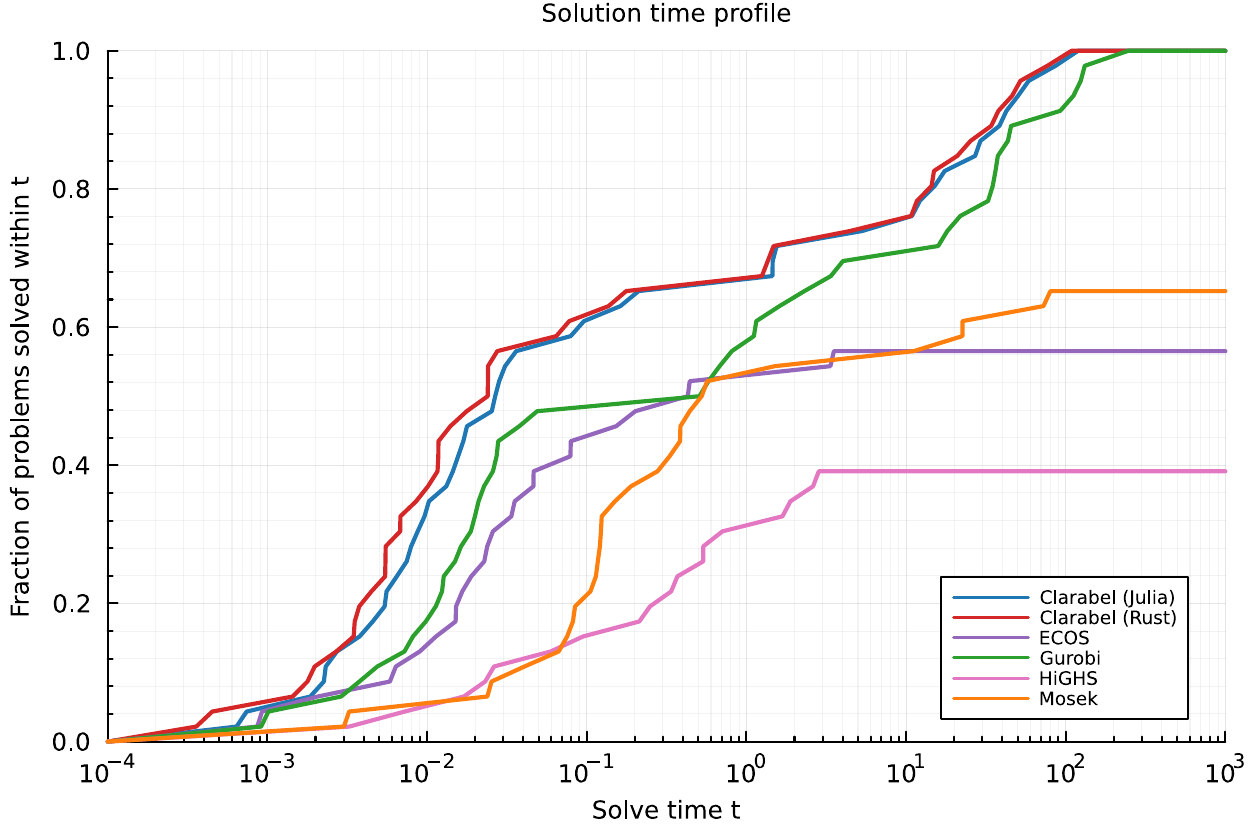}}
      \caption{Absolute performance profile}
      \label{fig:sslsq:absolute}
  \end{subfigure}
  \begin{subfigure}{1\textwidth}
    \centering
    \footnotesize
    \begin{tabular}{llcccccc}
  \hline
   &  & \textbf{ClarabelRs} & \textbf{Clarabel} & \textbf{ECOS} & \textbf{Gurobi} & \textbf{HiGHS} & \textbf{Mosek} \\\hline
  Shifted GM & Full Acc. & 1.0 & 1.06 & 7.12 & 1.79 & 21.5 & 6.31 \\
   & Low Acc. & 1.0 & 1.06 & 5.43 & 1.79 & 21.5 & 1.68 \\
  Failure Rate (\%) & Full Acc. & 0.0 & 0.0 & 43.5 & 0.0 & 60.9 & 34.8 \\
   & Low Acc. & 0.0 & 0.0 & 39.1 & 0.0 & 60.9 & 0.0 \\\hline
\end{tabular}

    \caption{Benchmark timings as shifted geometric mean and failure rates}
  \end{subfigure}
\end{figure}

\subsubsection{Constrained optimal control}

Finally, we consider finite-horizon constrained optimal control problems with quadratic objectives.  Problems of this type are of particular interest in embedded control systems, since the repeated online solution of such problems is the basis of the model predictive control (MPC) method~\cite{Borrelli:MPCbook}.     We consider a collection of 72 such problems taken from the benchmark collection of industrial and academic applications in \cite{Kouzoupis:2015}.  Problems in this collection are in the form
\MinProblem{\{x_i\}, \{y_i\}, \{u_i\}}{
\begin{aligned}[t] \nonumber
  \displaystyle\sum\limits_{i=0}^{N-1} &
  \begin{pmatrix}
    y_i - y^r_i \\
    u_i - u^r_i
  \end{pmatrix}
  \begin{pmatrix}
    Q_k & S_k \\ S_k^T & R_k
  \end{pmatrix}
  \begin{pmatrix}
    y_i - y^r_i\\
    u_i - u^r_i
  \end{pmatrix} \hfill 
  +   \begin{pmatrix}
    g^y_k \\ g_u^k 
  \end{pmatrix}\tpose 
  \begin{pmatrix}
    y_i - y^r_i\\
    u_i - u^r_i
  \end{pmatrix}  \\[1ex] 
  &+ (x_N - x^r_N)\tpose P(x_N - x^r_N) \\[2ex]
\end{aligned}
}
{
  \begin{gathered}[t]
    \\[-3ex]
  \begin{aligned}[t]
    &
    \left.
    \begin{aligned}
&x_{k+1} = A_k x_k + B_ku_k + f_k \\[0.5ex]
&y_{k}   = C_k x_k + D_ku_k + e_k \\[0.5ex]
&d_k^\ell \le M_k x_k + N_ku_k \le d_k^u \\[0.5ex]
&u_k \in \mathcal{U}_k,~y_k \in \mathcal{Y}_k \\[0.5ex]
    \end{aligned}  \quad \right\} \quad k = 0 \dots N-1 \\[0.5ex]
    &Tx_n \in \mathcal{\mathcal{T}},
  \end{aligned}
\end{gathered}
}

where the constraint sets $\mathcal{U}_k$, $\mathcal{Y}_k$ and $\mathcal{T}$ are interval constraints.   All problems have ${Q}_k \succeq 0$, ${R}_k \succeq 0$ and $P \succ 0$, which ensures that the problems are all convex QPs.  As is typical of optimal control problems for embedded systems, the dimension of the states $x_k$ and inputs $u_k$ are relatively small (max 12 and 4, respectively), with horizons $N$ up to 100.  

\paragraph{Results} Results for this benchmark set are shown in Figure~\ref{fig:mpc}.   Clarabel is the fastest solver overall, and is the only solver tested with a 100\% success rate in solving problems to full accuracy.

\captionsetup{labelfont=bf}
\begin{figure}[t]
  \centering
  \caption{\bf Performance profiles for the optimal control problem set}
  \label{fig:mpc}
  \begin{subfigure}[b]{0.49\textwidth}
      \centering
      {\includegraphics[width=\textwidth]{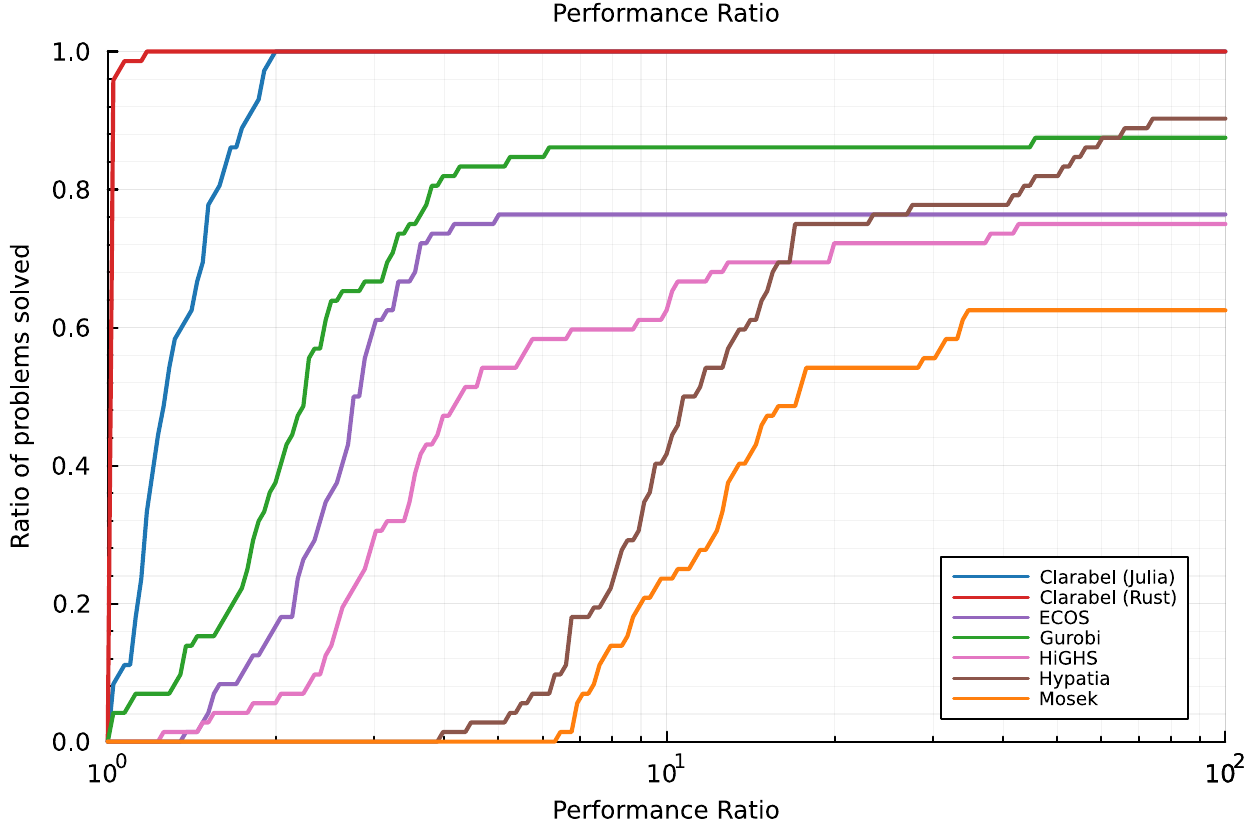}}
      \caption{Relative performance profile}
      \label{fig:mpc:relative}
  \end{subfigure}
  \hfill
  \begin{subfigure}[b]{0.49\textwidth}
      \centering
      {\includegraphics[width=\textwidth]{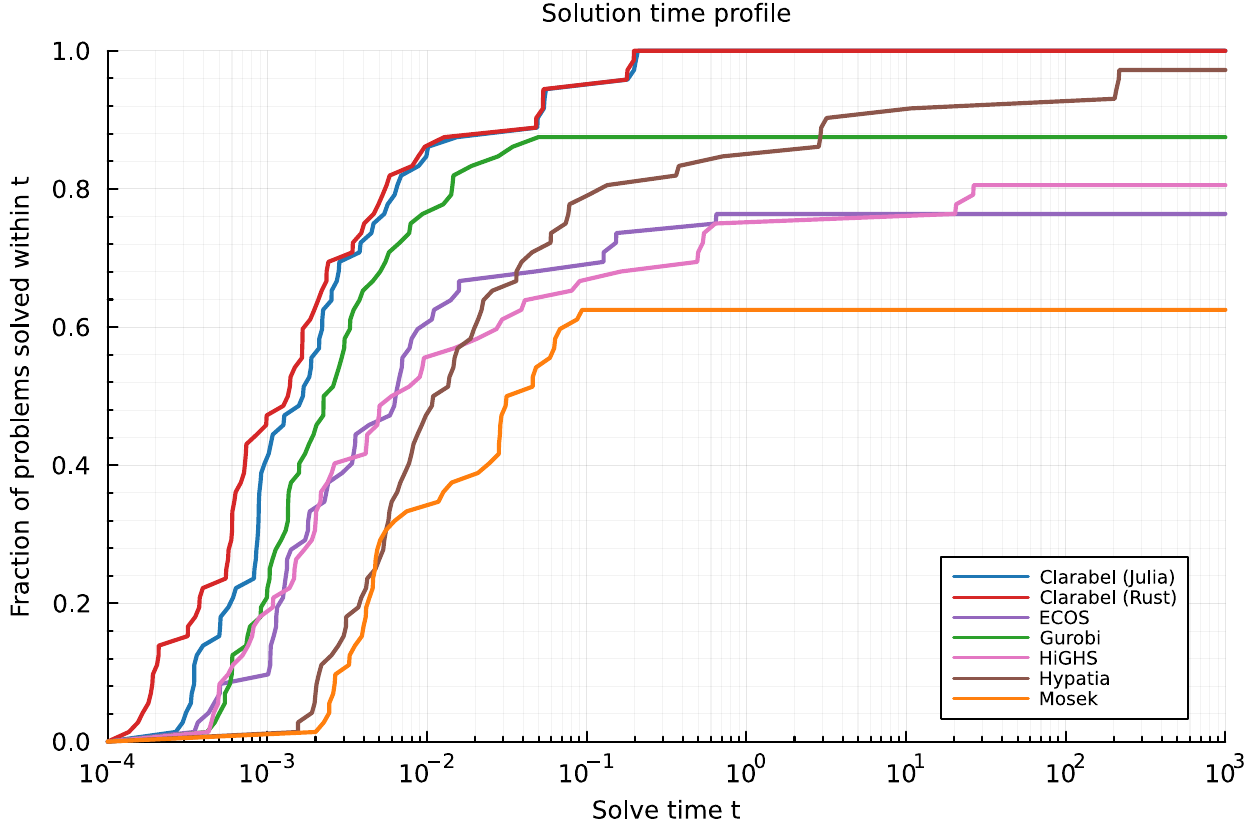}}
      \caption{Absolute performance profile}
      \label{fig:mpc:absolute}
  \end{subfigure}
  \begin{subfigure}{1\textwidth}
    \centering
    \footnotesize
    \begin{tabular}{llccccccc}
  \hline
   &  & \textbf{ClarabelRs} & \textbf{Clarabel} & \textbf{ECOS} & \textbf{Gurobi} & \textbf{HiGHS} & \textbf{Hypatia} & \textbf{Mosek} \\\hline
  Shifted GM & Full Acc. & 1.0 & 1.06 & 197.57 & 70.48 & 185.45 & 54.95 & 511.01 \\
   & Low Acc. & 1.0 & 1.06 & 14.97 & 60.07 & 185.45 & 37.36 & 8.26 \\
  Failure Rate (\%) & Full Acc. & 0.0 & 0.0 & 23.6 & 12.5 & 19.4 & 2.8 & 37.5 \\
   & Low Acc. & 0.0 & 0.0 & 2.8 & 11.1 & 19.4 & 0.0 & 0.0 \\\hline
\end{tabular}

    \caption{Benchmark timings as shifted geometric mean and failure rates}
  \end{subfigure}
\end{figure}

\subsection{Benchmark problems with linear objectives}

In this section we present benchmark results for optimization problems \emph{without} quadratic objective terms, i.e.\ with $P = 0$ in \ref{eqn:primal}.   We again consider example problems taken or generated from standard open-source problem collections and covering a wide range of problem dimensions.   Our test set covers cases with constraints on the positive orthant (i.e.\ linear programs), as well as both second-order and exponential cone programs.

\subsubsection{NETLIB LP problems}
We first consider linear programming (LP) problems taken from the NETLIB collection, a standard collection of benchmark LPs~\cite{netlib:1987, netlib}.   Our benchmark test set includes 117 feasible and 29 infeasible test cases, representing all NETLIB LP problem instances with source files not exceeding 2.5 MB.   

Results for the feasible and infeasible test sets are shown in Figures~\ref{fig:netlib_feasible} and \ref{fig:netlib_infeasible}, respectively.   For these cases Clarabel has performance broadly similar to both Mosek and Gurobi for the feasible set, and somewhat slower for the infeasible set.   This result is to be expected since most of the potential performance advantage of our method arises from improved handling of quadratic objectives, but illustrates that our implementation is, in the LP case, still broadly comparable.

\captionsetup{labelfont=bf}
\begin{figure}[ht]
  \centering
  \caption{\bf Performance profiles for the NETLIB Feasible LP problem set}
  \label{fig:netlib_feasible}
  \begin{subfigure}[b]{0.49\textwidth}
      \centering
      {\includegraphics[width=\textwidth]{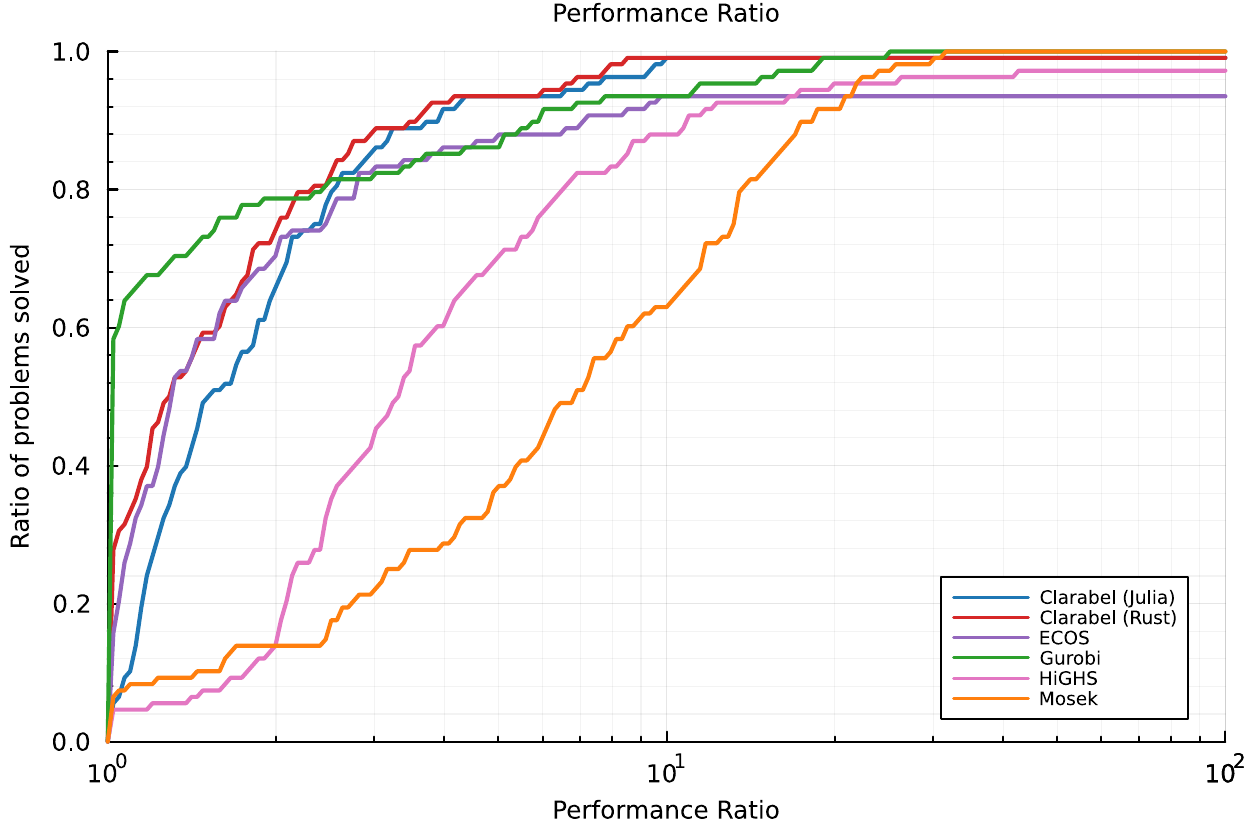}}
      \caption{Relative performance profile}
      \label{fig:netlib_feasible:relative}
  \end{subfigure}
  \hfill
  \begin{subfigure}[b]{0.49\textwidth}
      \centering
      {\includegraphics[width=\textwidth]{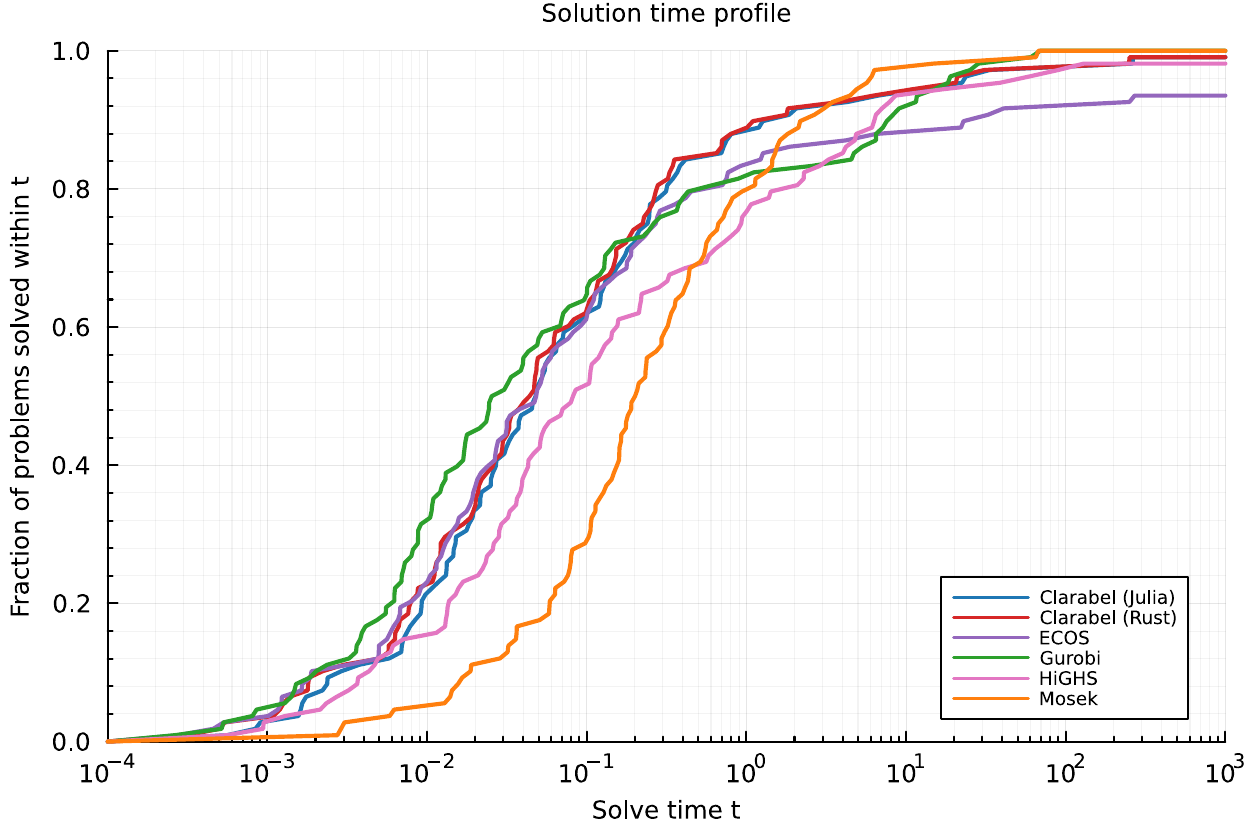}}
      \caption{Absolute performance profile}
      \label{fig:netlib_feasible:absolute}
  \end{subfigure}
  \begin{subfigure}{1\textwidth}
    \centering
    \footnotesize
    \begin{tabular}{llcccccc}
  \hline
   &  & \textbf{ClarabelRs} & \textbf{Clarabel} & \textbf{ECOS} & \textbf{Gurobi} & \textbf{HiGHS} & \textbf{Mosek} \\\hline
  Shifted GM & Full Acc. & 1.0 & 1.04 & 2.13 & 1.3 & 1.79 & 1.2 \\
   & Low Acc. & 1.0 & 1.04 & 1.08 & 1.3 & 1.79 & 1.2 \\
  Failure Rate (\%) & Full Acc. & 0.9 & 0.9 & 6.5 & 0.0 & 1.9 & 0.0 \\
   & Low Acc. & 0.9 & 0.9 & 0.9 & 0.0 & 1.9 & 0.0 \\\hline
\end{tabular}

    \caption{Benchmark timings as shifted geometric mean and failure rates}
  \end{subfigure}
\end{figure}

\captionsetup{labelfont=bf}
\begin{figure}[ht]
  \centering
  \caption{\bf Performance profiles for the NETLIB Infeasible LP problem set}
  \label{fig:netlib_infeasible}
  \begin{subfigure}[b]{0.49\textwidth}
      \centering
      {\includegraphics[width=\textwidth]{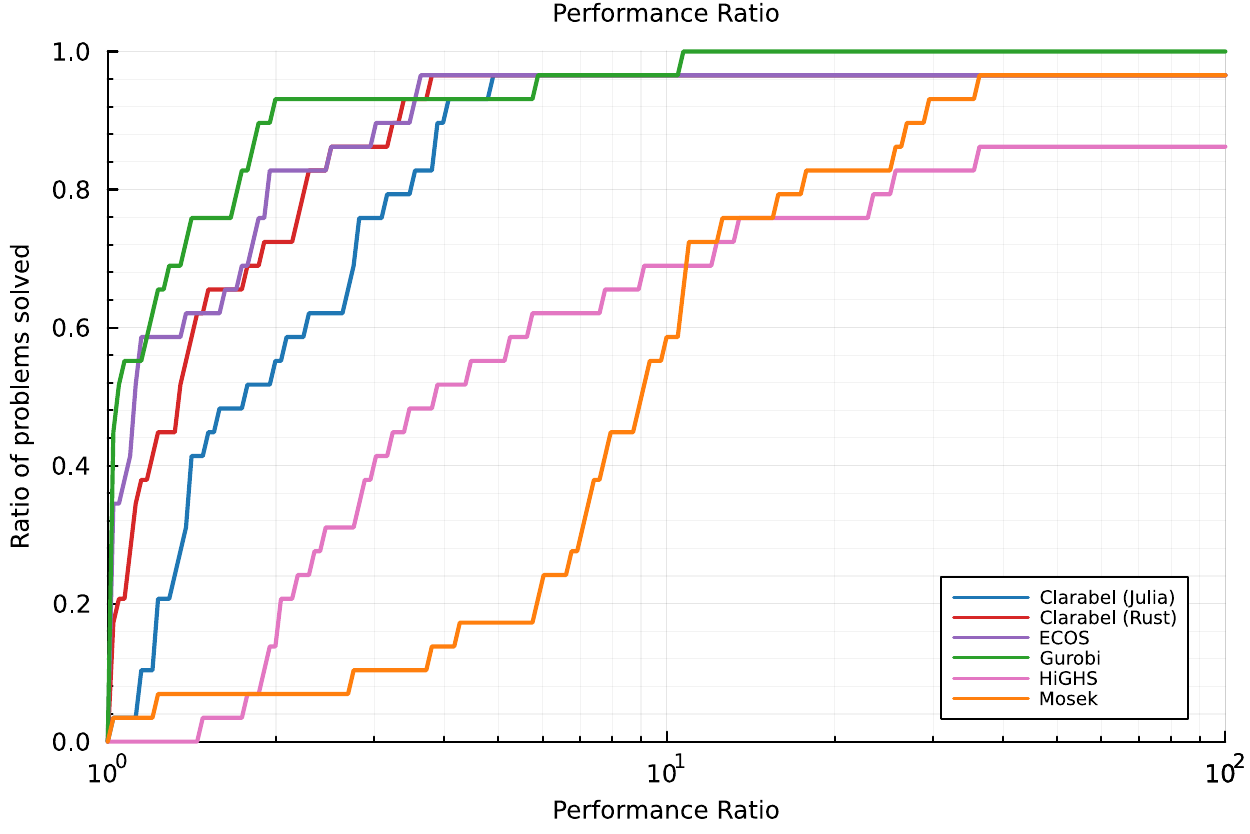}}
      \caption{Relative performance profile}
      \label{fig:netlib_infeasible:relative}
  \end{subfigure}
  \hfill
  \begin{subfigure}[b]{0.49\textwidth}
      \centering
      {\includegraphics[width=\textwidth]{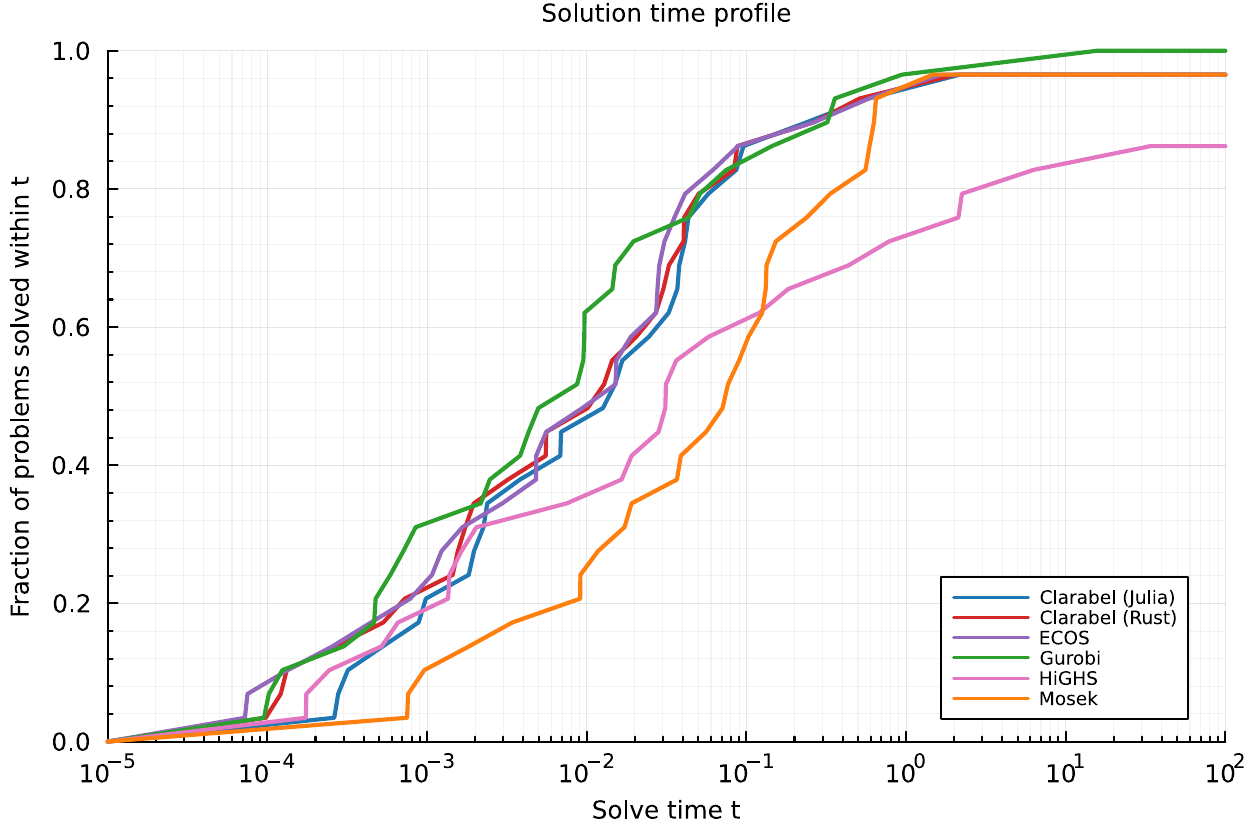}}
      \caption{Absolute performance profile}
      \label{fig:netlib_infeasible:absolute}
  \end{subfigure}
  \begin{subfigure}{1\textwidth}
    \centering
    \footnotesize
    \begin{tabular}{llcccccc}
  \hline
   &  & \textbf{ClarabelRs} & \textbf{Clarabel} & \textbf{ECOS} & \textbf{Gurobi} & \textbf{HiGHS} & \textbf{Mosek} \\\hline
  Shifted GM & Full Acc. & 1.89 & 1.92 & 1.85 & 1.0 & 12.19 & 2.49 \\
   & Low Acc. & 1.89 & 1.92 & 1.85 & 1.0 & 12.19 & 2.49 \\
  Failure Rate (\%) & Full Acc. & 3.4 & 3.4 & 3.4 & 0.0 & 13.8 & 3.4 \\
   & Low Acc. & 3.4 & 3.4 & 3.4 & 0.0 & 13.8 & 3.4 \\\hline
\end{tabular}

    \caption{Benchmark timings as shifted geometric mean and failure rates}
  \end{subfigure}
\end{figure}

\subsubsection{CBLIB exponential cone problems}

We consider a collection of 39 exponential cone programs from the CBLIB benchmark collection~\cite{CBLIB} in order to test performance on nonsymmetric cone programs.   For these problems our implementation has broadly similar performance to Mosek, with slightly faster solve times in the Rust implementation despite a generally higher iteration count.

\captionsetup{labelfont=bf}
\begin{figure}[ht]
  \centering
  \caption{\bf Performance profiles for the CBLIB Exponential Cone problem set}
  \label{fig:cblib_exp}
  \begin{subfigure}[b]{0.49\textwidth}
      \centering
      {\includegraphics[width=\textwidth]{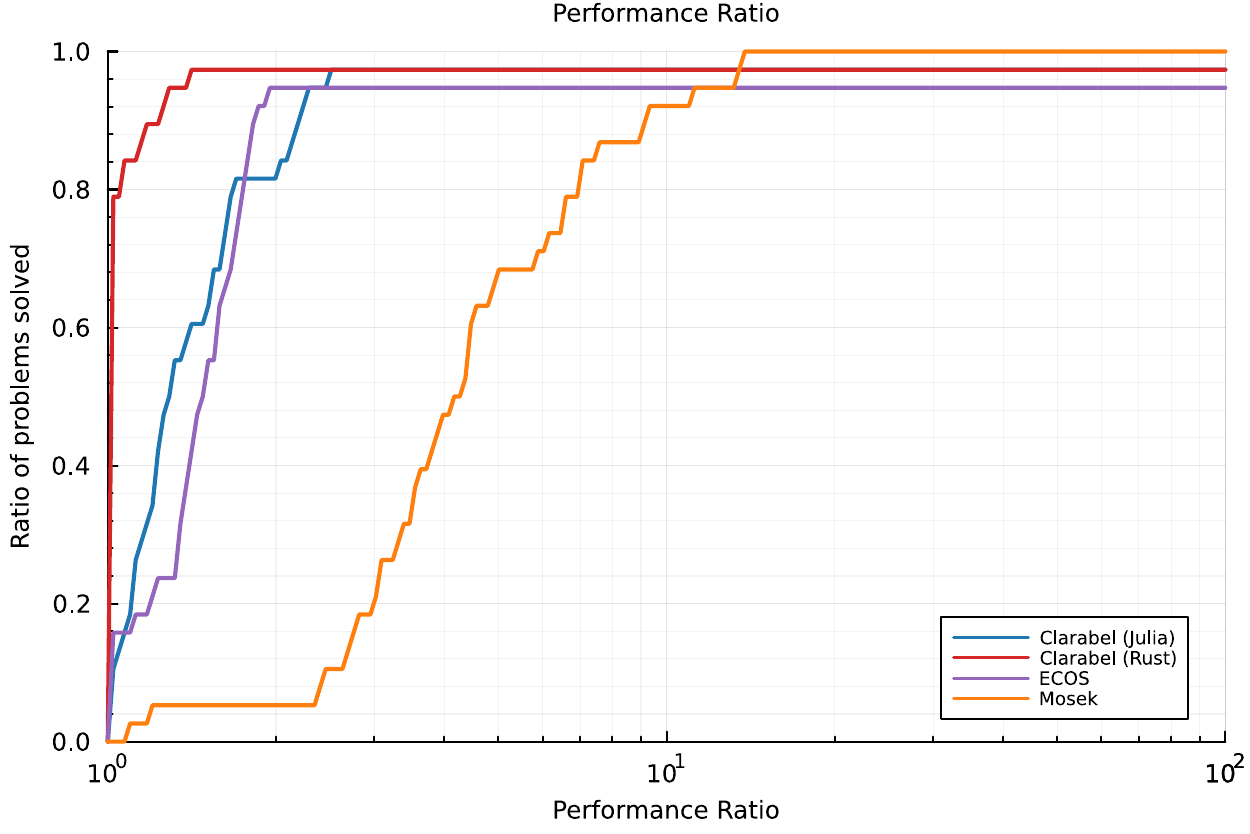}}
      \caption{Relative performance profile}
      \label{fig:cblib_exp:relative}
  \end{subfigure}
  \hfill
  \begin{subfigure}[b]{0.49\textwidth}
      \centering
      {\includegraphics[width=\textwidth]{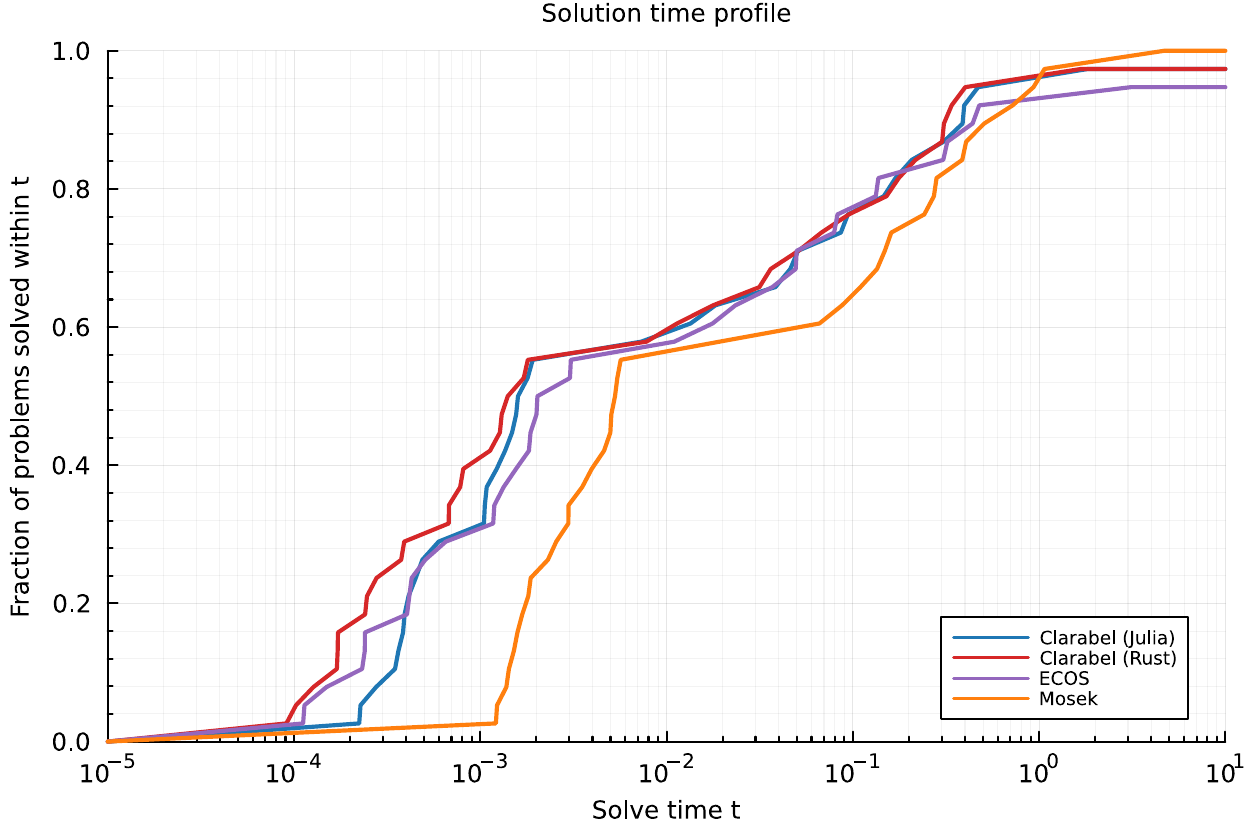}}
      \caption{Absolute performance profile}
      \label{fig:cblib_exp:absolute}
  \end{subfigure}
  \begin{subfigure}{1\textwidth}
    \centering
    \footnotesize
    \begin{tabular}{llcccc}
  \hline
   &  & \textbf{ClarabelRs} & \textbf{Clarabel} & \textbf{ECOS} & \textbf{Mosek} \\\hline
  Shifted GM & Full Acc. & 1.45 & 1.49 & 2.68 & 1.0 \\
   & Low Acc. & 1.0 & 1.08 & 3.28 & 2.14 \\
  Failure Rate (\%) & Full Acc. & 2.6 & 2.6 & 5.3 & 0.0 \\
   & Low Acc. & 0.0 & 0.0 & 2.6 & 0.0 \\\hline
\end{tabular}

    \caption{Benchmark timings as shifted geometric mean and failure rates}
  \end{subfigure}
\end{figure}

\clearpage

\subsubsection{Optimal power flow}

We next consider a collection of optimal power flow problems based on power networks from IEEE PLS PGLib-OPF benchmark library \cite{PGLIB:2019} and constructed using the \texttt{PowerModels.jl} benchmark test framework \cite{PowerModels:2018}.   This framework allows for the generation of optimal power flow problems with various modelling assumptions and convex relaxations applied \cite{PowerModels:survey1,PowerModels:survey2}.   We consider in particular linear programming problems based on linearized (i.e.\ direct current (DC)) power flow models \cite{PowerModels:LPcase}, and SOCPs arising from second-order cone relaxations of AC models \cite{PowerModels:SOCcase}.   Results from these benchmarks are shown in Figures~\ref{fig:opf_lp} and \ref{fig:opf_socp}, respectively.  

For both of these test sets Clarabel outperforms the other solvers tested, albeit with a slightly higher failure rate for the LP benchmark tests relative to Gurobi (i.e.\ 3.3\% vs 0\% at full accuracy).   For the SOCP benchmark tests in particular our success rate is substantially better than the other solvers tested.   All solvers in our benchmark group struggled to some extent in solve large scale SOCPs to full accuracy.   Our Rust implementation is able to solve to at least its reduced accuracy level for more than 99\% of test cases though, a success rate considerably higher than ECOS or Mosek, both of which failed even at reduced accuracy on more than 50\% of cases.   

We note also that our success rates differ slightly between our Julia and Rust implementations, even though the implementation of our algorithm is (nearly) identical between the cases.  We believe that this difference is attributable to minor differences in compiled code vectorizations and optimizations, which in some very difficult problems lead to slight differences in behavior at high accuracy.

\captionsetup{labelfont=bf}
\begin{figure}[ht]
  \centering
  \caption{\bf Performance profiles for the LP Optimal Power Flow problem set}
  \label{fig:opf_lp}
  \begin{subfigure}[b]{0.49\textwidth}
      \centering
      {\includegraphics[width=\textwidth]{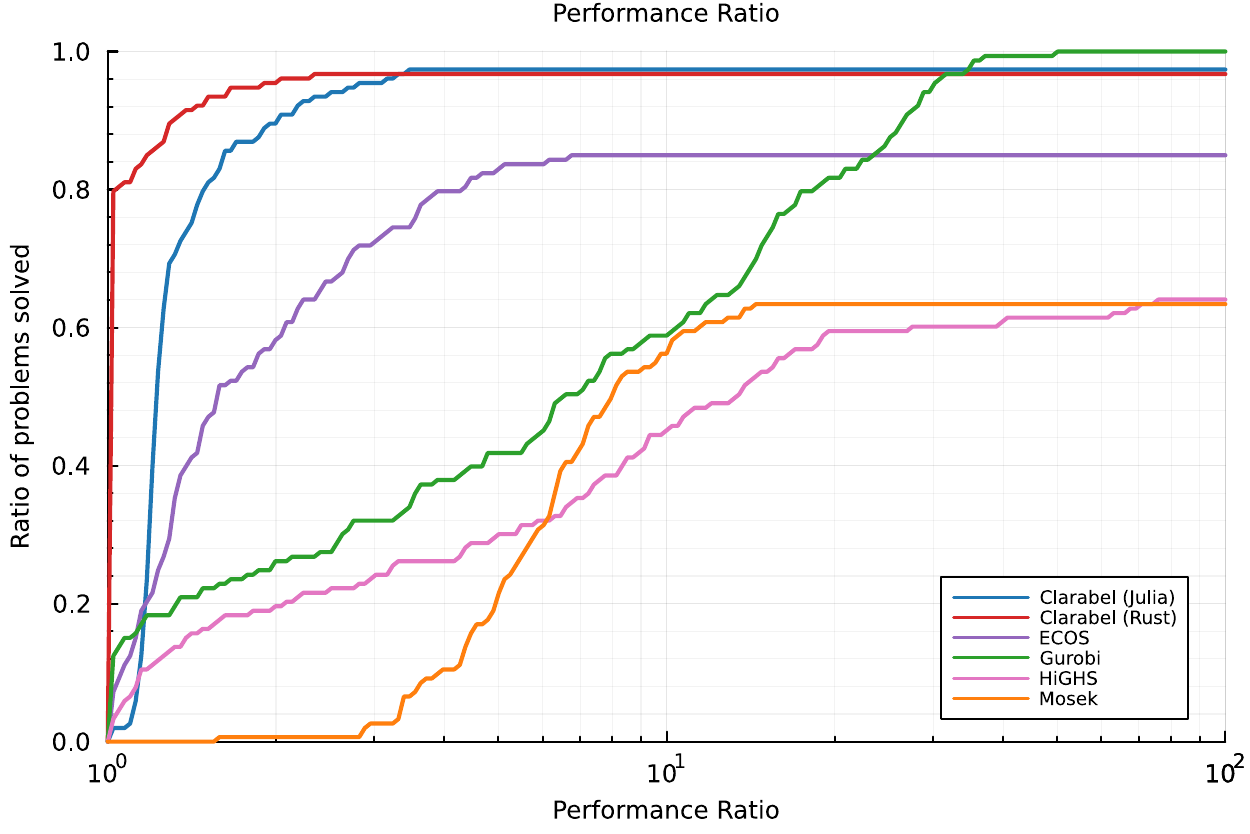}}
      \caption{Relative performance profile}
      \label{fig:opf_lp:relative}
  \end{subfigure}
  \hfill
  \begin{subfigure}[b]{0.49\textwidth}
      \centering
      {\includegraphics[width=\textwidth]{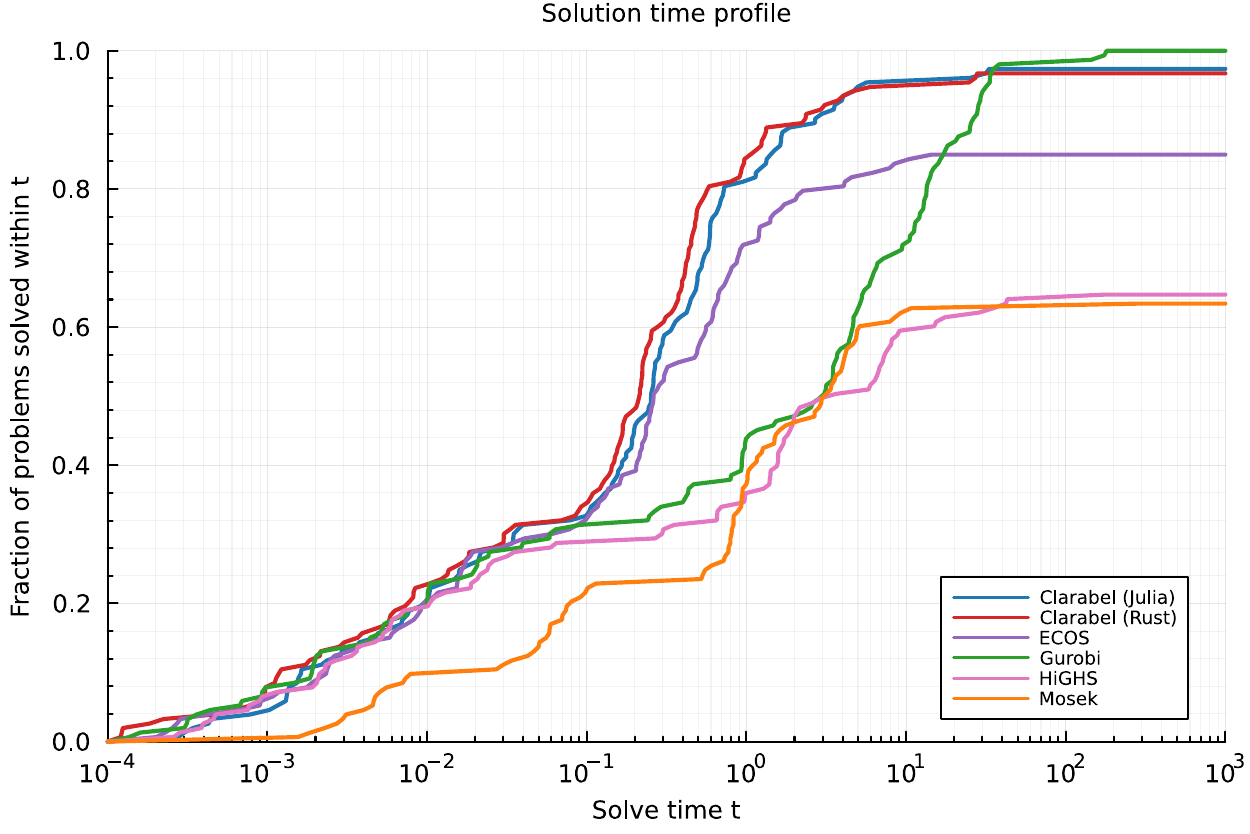}}
      \caption{Absolute performance profile}
      \label{fig:opf_lp:absolute}
  \end{subfigure}
  \begin{subfigure}{1\textwidth}
    \centering
    \footnotesize
    \begin{tabular}{llcccccc}
  \hline
   &  & \textbf{ClarabelRs} & \textbf{Clarabel} & \textbf{ECOS} & \textbf{Gurobi} & \textbf{HiGHS} & \textbf{Mosek} \\\hline
  Shifted GM & Full Acc. & 1.0 & 1.04 & 3.27 & 4.48 & 17.73 & 17.58 \\
   & Low Acc. & 1.0 & 1.07 & 2.89 & 5.57 & 22.04 & 21.85 \\
  Failure Rate (\%) & Full Acc. & 3.3 & 2.6 & 15.0 & 0.0 & 35.3 & 36.6 \\
   & Low Acc. & 1.3 & 0.7 & 9.8 & 0.0 & 35.3 & 36.6 \\\hline
\end{tabular}

    \caption{Benchmark timings as shifted geometric mean and failure rates}
  \end{subfigure}
\end{figure}

\captionsetup{labelfont=bf}
\begin{figure}[ht]
  \centering
  \caption{\bf Performance profiles for the SOCP Optimal Power Flow problem set}
  \label{fig:opf_socp}
  \begin{subfigure}[b]{0.49\textwidth}
      \centering
      {\includegraphics[width=\textwidth]{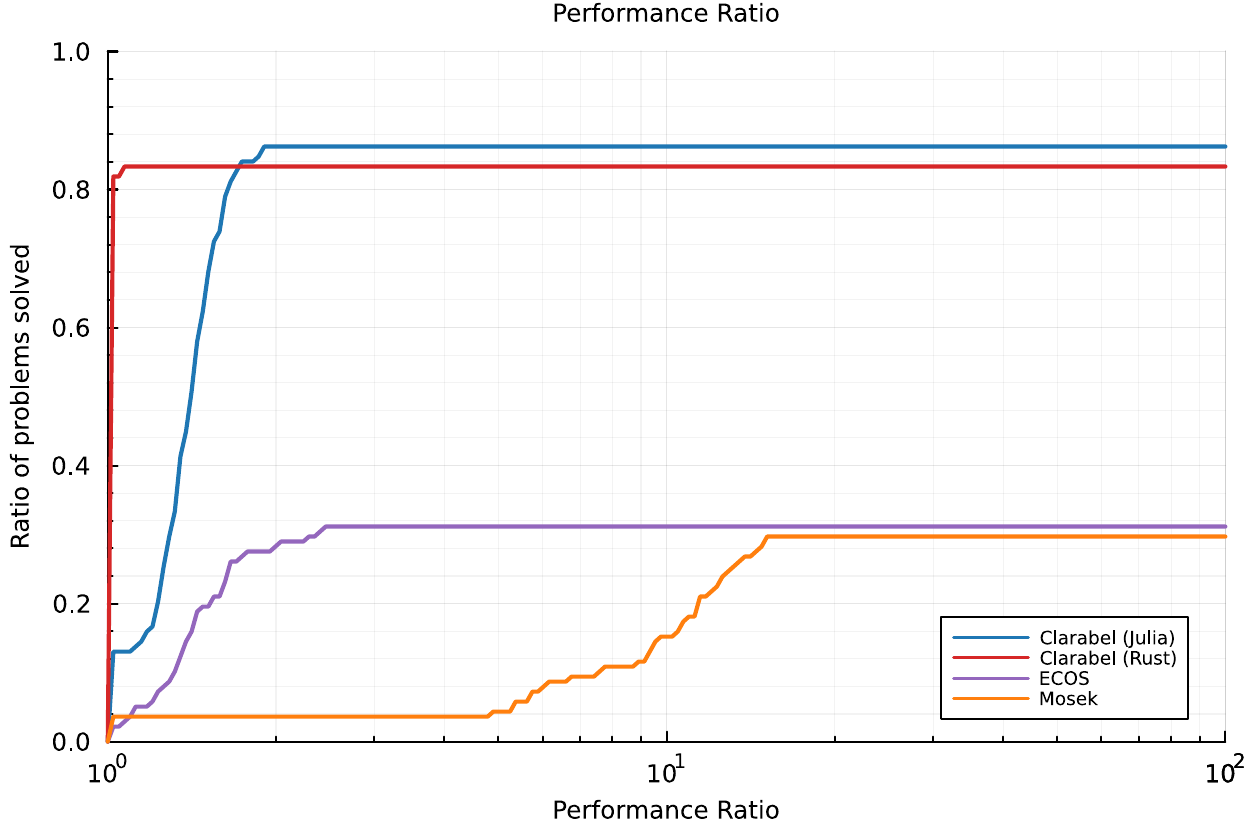}}
      \caption{Relative performance profile}
      \label{fig:opf_socp:relative}
  \end{subfigure}
  \hfill
  \begin{subfigure}[b]{0.49\textwidth}
      \centering
      {\includegraphics[width=\textwidth]{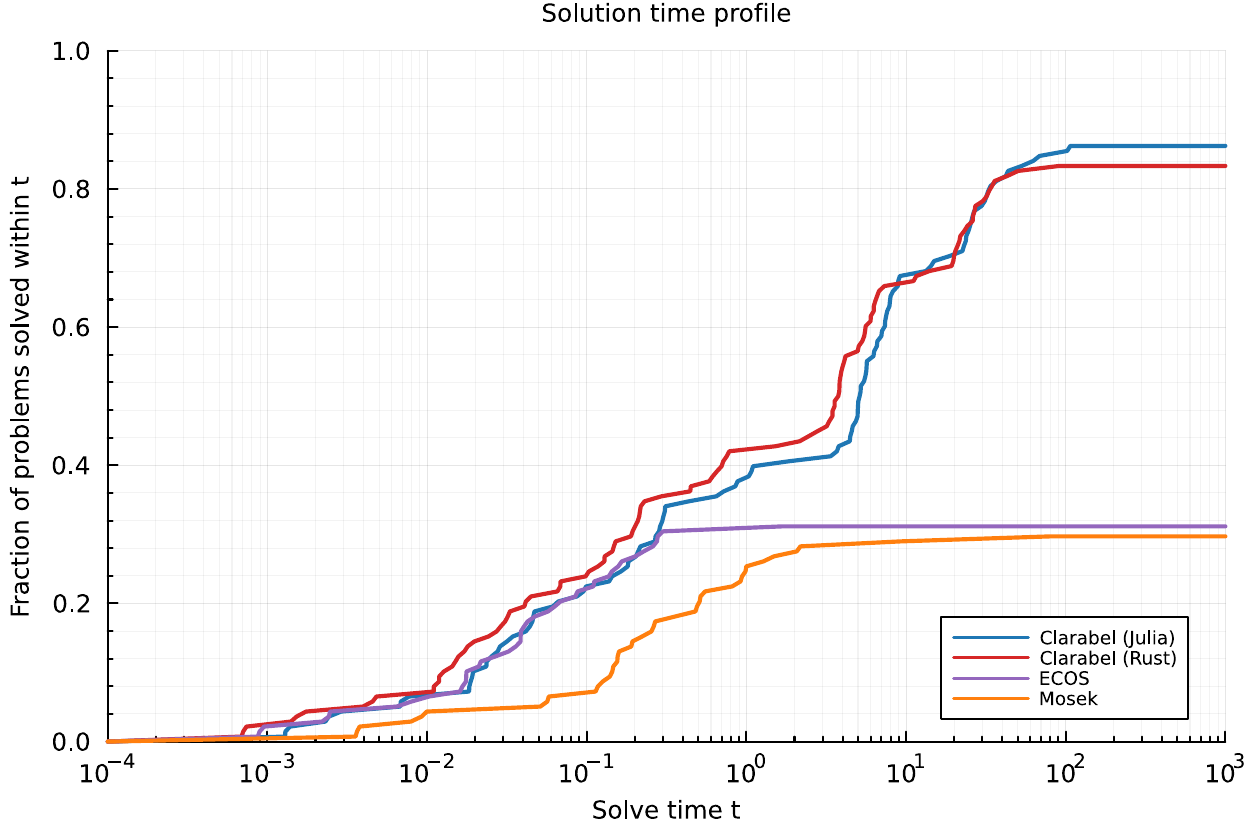}}
      \caption{Absolute performance profile}
      \label{fig:opf_socp:absolute}
  \end{subfigure}
  \begin{subfigure}{1\textwidth}
    \centering
    \footnotesize
    \begin{tabular}{llcccc}
  \hline
   &  & \textbf{ClarabelRs} & \textbf{Clarabel} & \textbf{ECOS} & \textbf{Mosek} \\\hline
  Shifted GM & Full Acc. & 1.0 & 1.07 & 8.25 & 10.02 \\
   & Low Acc. & 1.0 & 1.33 & 8.62 & 20.48 \\
  Failure Rate (\%) & Full Acc. & 16.7 & 13.8 & 68.8 & 70.3 \\
   & Low Acc. & 0.7 & 2.2 & 55.8 & 70.3 \\\hline
\end{tabular}

    \caption{Benchmark timings as shifted geometric mean and failure rates}
  \end{subfigure}
\end{figure}

\subsection{Semidefinite program benchmarks}

\captionsetup{labelfont=bf}
\begin{figure}[th]
  \centering
  \caption{\bf Performance profiles for the SDPLIB Semidefinite Programming problem set}
  \label{fig:sdplib}
  \begin{subfigure}[b]{0.49\textwidth}
      \centering
      {\includegraphics[width=\textwidth]{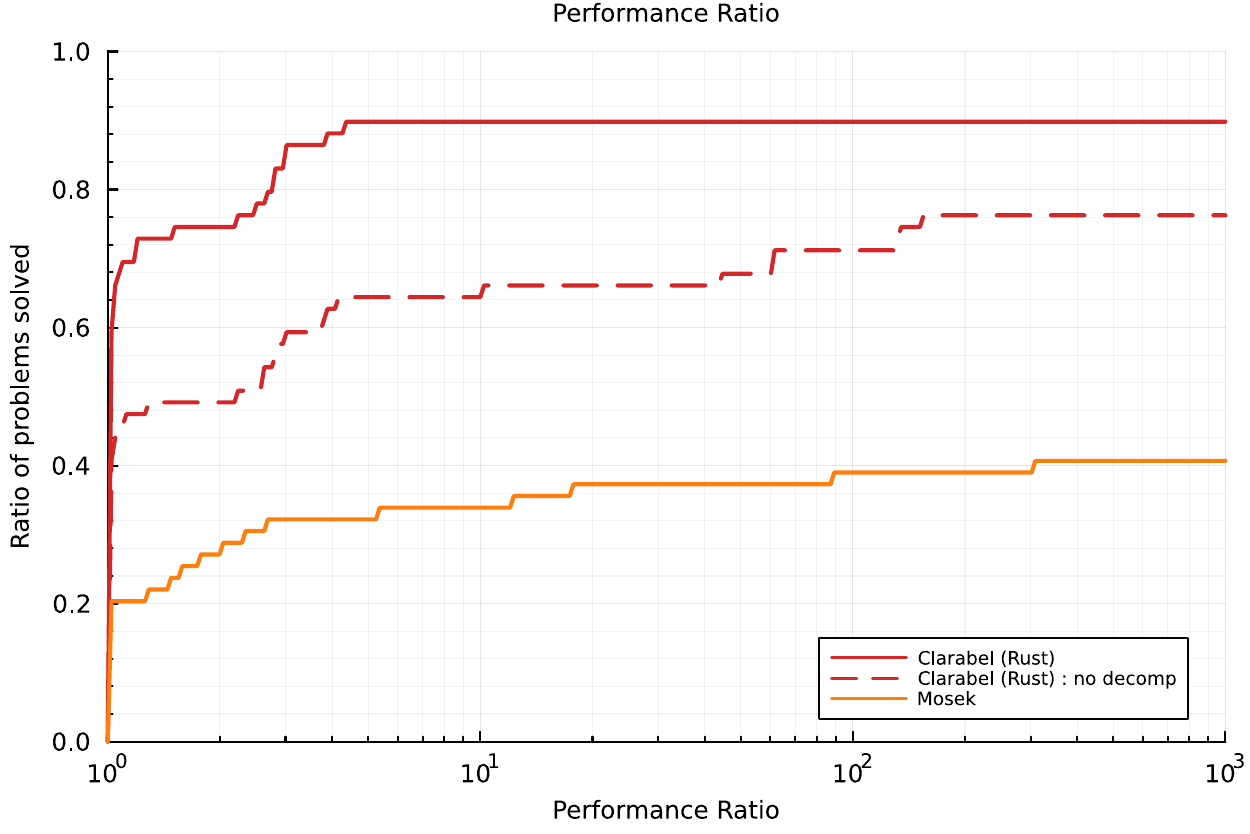}}
      \caption{Relative performance profile}
      \label{fig:sdplib:relative}
  \end{subfigure}
  \hfill
  \begin{subfigure}[b]{0.49\textwidth}
      \centering
      {\includegraphics[width=\textwidth]{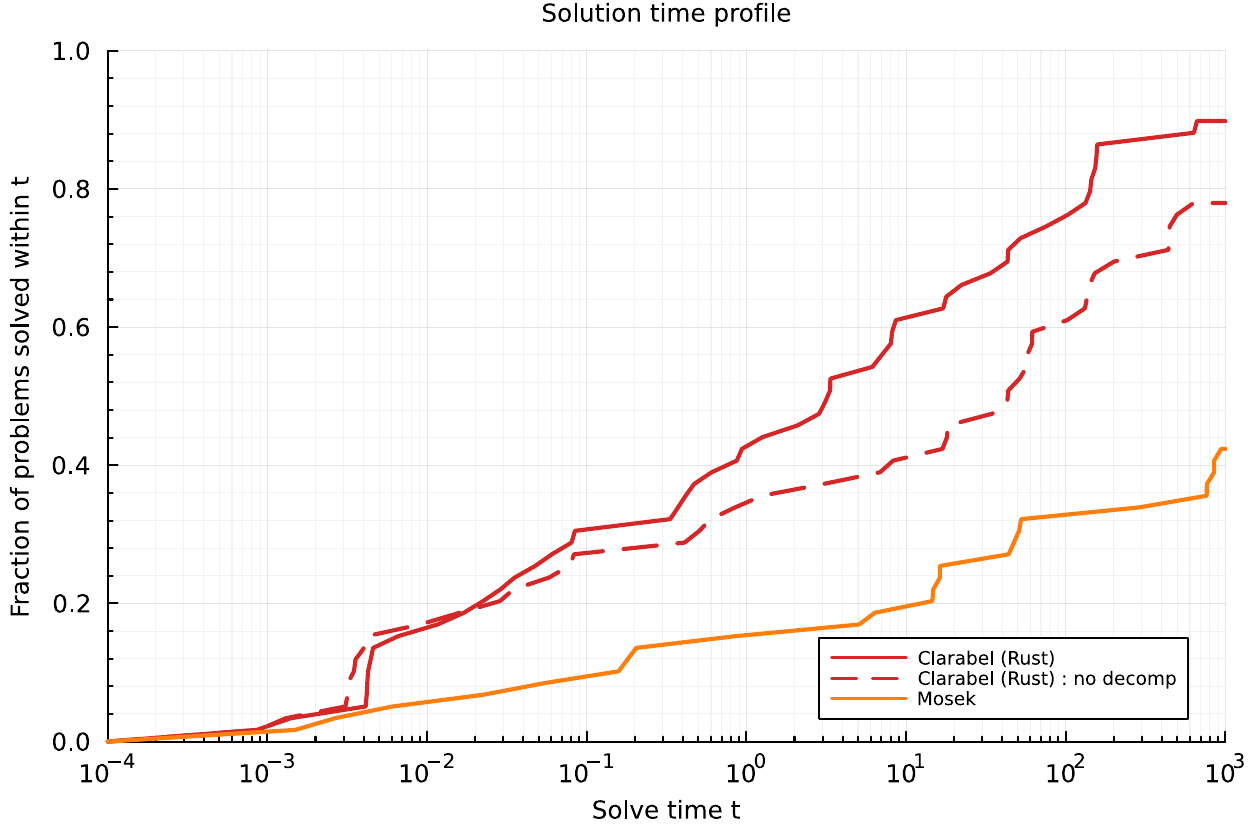}}
      \caption{Absolute performance profile}
      \label{fig:sdplib:absolute}
  \end{subfigure}
  \begin{subfigure}{1\textwidth}
    \centering
    \footnotesize
    \newcommand{\nochordal}{
    \begin{tabular}{@{}c@{}}\textbf{ClarabelRs}\\(no decomp)\end{tabular}
}

\newcommand{\raisesolverbox}[1]{
    \begin{tabular}{@{}c@{}}\textbf{#1}\\~\end{tabular}
}

\begin{tabular}{llccc}
  \hline
   &  & \raisesolverbox{ClarabelRs} & {\nochordal} & \raisesolverbox{Mosek} \\\hline
  Shifted GM & Full Acc. & 1.0 & 2.53 & 9.95 \\
   & Low Acc. & 1.0 & 2.17 & 1.26 \\
  Failure Rate (\%) & Full Acc. & 10.2 & 22.0 & 57.6 \\
   & Low Acc. & 1.7 & 10.2 & 1.7 \\\hline
\end{tabular}

    \caption{Benchmark timings as shifted geometric mean and failure rates}
  \end{subfigure}
\end{figure}

Finally, we test performance on large-scale SDPs using the SDPLIB benchmark collection~\cite{SDPLIB}.   We consider a collection of 64 problems from this collection, where problems from the collection are included if they are known to be feasible and could be solved on our testing platform using the reference solver Mosek.   Since some of these problems are large, we set a time limit of 30 minutes for each problem.   

For this set of benchmarks we compare only our Rust implementation of Clarabel with Mosek and use the supernodal factorization method of \cite{faer}.  We provide results for Clarabel both with and without the chordal decomposition described in \S\ref{sec:chordal_decomposition} enabled.   

For this set of benchmarks we find that Clarabel outperforms Mosek overall, particularly when considering problems solved to full accuracy solutions.   Clarabel also outperforms Mosek, sometimes very substantially, for those test cases with a considerable amount of sparsity that are amenable to chordal decomposition.    We include in Appendix~\ref{appendix:full_results} a detailed breakdown of the results for this benchmark set, including performance both with and without chordal decomposition enabled.   For problems in which no chordal structure could be identified, there is very little difference in performance between our two solver configurations.   However, some problems (e.g.\ the series of problems `ARCH0'--`ARCH8') show improvements in computation time of nearly 100x when chordal decomposition is enabled.   We note that further improvements are likely possible in our implementation of chordal decomposition, particularly if more sophisticated clique-merging scoring methods are employed.

\section{Conclusions}

We have presented a novel interior-point solver for conic optimization problems with quadratic objectives.   Our method uses a homogeneous embedding inspired by previous work on monotone complementarity problems, but not previously applied to interior-point conic optimization in any widely available solver.   We have shown that our method is competitive with state-of-the-art solvers for a wide range of problem classes, and in particular outperforms state-of-the-art solvers in problems with quadratic objectives (QPs), large-scale SOCPs, and SDPs with significant sparsity structure.   

Our implementation of Clarabel is available as open-source software in both Rust and Julia, with several other language interfaces, and is available as a standard solver in the CVXPY modelling package.   Clarabel already has growing base of both academic and industrial users and has been downloaded several million times since its initial release.

\clearpage

\printbibliography

\appendix
\section{Detailed benchmark results}\label{appendix:full_results}

\clearpage

\captionsetup{labelfont=bf}
\centering
\renewcommand{\detailtablecaption}{\bf Solve times and iteration counts for the Maros-Meszaros problem set}
\begin{landscape}
\scriptsize
\begin{longtable}{lcccc||ccc||ccc||ccc||}
\caption{\detailtablecaption}
\\
 & &  & & & \multicolumn{3}{c||}{\underline{iterations}}& \multicolumn{3}{c||}{\underline{time per iteration(s)}}& \multicolumn{3}{c||}{\underline{total time (s)}}\\[2ex] 
Problem & vars. & cons. & nnz(A) & nnz(P)  & ClarabelRs & ECOS & Mosek & ClarabelRs & ECOS & Mosek & ClarabelRs & ECOS & Mosek\\[1ex]
\hline
\endhead
\sc{TAME} & 3 & 2 & 4 & 3 &   5 &   6 &  \winner  3 &   1.16e-05 &  \winner  7.36e-06 &   0.000154 &   5.81e-05 &  \winner  4.41e-05 &   0.000461\\ 
\sc{HS21} & 5 & 2 & 6 & 2 &  \winner  9 &   11 &   11 &  \winner  6.19e-06 &   6.63e-06 &   0.000105 &  \winner  5.57e-05 &   7.29e-05 &   0.00115\\ 
\sc{HS35} & 4 & 3 & 6 & 5 &  \winner  7 &   9 &   9 &  \winner  7.64e-06 &   9.51e-06 &   0.000105 &  \winner  5.35e-05 &   8.56e-05 &   0.000944\\ 
\sc{HS35MOD} & 4 & 3 & 6 & 5 &  \winner  12 &   15 &   14 &  \winner  5.37e-06 &   7.67e-06 &   9.91e-05 &  \winner  6.45e-05 &   0.000115 &   0.00139\\ 
\sc{QPTEST} & 5 & 2 & 7 & 3 &  \winner  8 &   9 &  \winner  8 &  \winner  8.01e-06 &   9.51e-06 &   0.000158 &  \winner  6.41e-05 &   8.56e-05 &   0.00126\\ 
\sc{HS51} & 3 & 5 & 7 & 7 &  \winner  0 &   9 &   9 &  \winner     0 &   8.83e-06 &   0.000104 &  \winner  3.2e-05 &   7.95e-05 &   0.000934\\ 
\sc{HS52} & 3 & 5 & 7 & 7 &  \winner  0 &   7 &   5 &  \winner     0 &   1.06e-05 &   0.000121 &  \winner  5.51e-05 &   7.43e-05 &   0.000607\\ 
\sc{ZECEVIC2} & 6 & 2 & 8 & 1 &   8 &   8 &  \winner  7 &   8.3e-06 &  \winner  7.49e-06 &   0.000115 &   6.64e-05 &  \winner  6e-05 &   0.000806\\
\end{longtable}

\end{landscape}

\captionsetup{labelfont=bf}
\centering
\renewcommand{\detailtablecaption}{\bf Solve times and iteration counts for the Optimal Control problem set}
\begin{landscape}
\scriptsize
\begin{longtable}{lcccc||ccc||ccc||ccc||}
\caption{\detailtablecaption}
\\
 & &  & & & \multicolumn{3}{c||}{\underline{iterations}}& \multicolumn{3}{c||}{\underline{time per iteration(s)}}& \multicolumn{3}{c||}{\underline{total time (s)}}\\[2ex] 
Problem & vars. & cons. & nnz(A) & nnz(P)  & ClarabelRs & ECOS & Mosek & ClarabelRs & ECOS & Mosek & ClarabelRs & ECOS & Mosek\\[1ex]
\hline
\endhead
\sc{pendulum\_3} & 28 & 23 & 83 & 23 &  \winner  5 &   11 &   11 &  \winner  2.71e-05 &   3.17e-05 &   0.000182 &  \winner  0.000136 &   0.000349 &   0.002\\ 
\sc{fiordosExample\_1} & 32 & 27 & 67 & 18 &  \winner  8 &   12 &   12 &  \winner  1.93e-05 &   3.05e-05 &   0.000188 &  \winner  0.000155 &   0.000367 &   0.00226\\ 
\sc{fiordosExample\_3} & 32 & 27 & 67 & 18 &  \winner  8 &   - &   - &  \winner  2.07e-05 &   - &   - &  \winner  0.000166 &   - &   -\\ 
\sc{forcesExample\_1} & 38 & 29 & 83 & 28 &  \winner  9 &   14 &   11 &  \winner  2.07e-05 &   3.59e-05 &   0.000235 &  \winner  0.000186 &   0.000502 &   0.00258\\ 
\sc{forcesExample\_2} & 38 & 29 & 83 & 27 &  \winner  8 &   13 &   10 &  \winner  2.38e-05 &   3.57e-05 &   0.000244 &  \winner  0.00019 &   0.000464 &   0.00244\\ 
\sc{forcesExample\_3} & 42 & 32 & 92 & 31 &  \winner  9 &   12 &   11 &  \winner  2.32e-05 &   4.21e-05 &   0.00024 &  \winner  0.000209 &   0.000506 &   0.00264\\ 
\sc{toyExample\_1} & 42 & 32 & 102 & 31 &  \winner  7 &   - &   16 &  \winner  2.57e-05 &   - &   0.000318 &  \winner  0.00018 &   - &   0.00509\\ 
\sc{toyExample\_3} & 42 & 32 & 102 & 31 &  \winner  9 &   - &   - &  \winner  2.31e-05 &   - &   - &  \winner  0.000208 &   - &   -\\
\end{longtable}

\end{landscape}

\captionsetup{labelfont=bf}
\centering
\renewcommand{\detailtablecaption}{\bf Solve times and iteration counts for the SuiteSparse least-squares problem set}
\begin{landscape}
\scriptsize
\begin{longtable}{lcccc||ccc||ccc||ccc||}
\caption{\detailtablecaption}
\\
 & &  & & & \multicolumn{3}{c||}{\underline{iterations}}& \multicolumn{3}{c||}{\underline{time per iteration(s)}}& \multicolumn{3}{c||}{\underline{total time (s)}}\\[2ex] 
Problem & vars. & cons. & nnz(A) & nnz(P)  & ClarabelRs & ECOS & Mosek & ClarabelRs & ECOS & Mosek & ClarabelRs & ECOS & Mosek\\[1ex]
\hline
\endhead
\sc{NYPA\_Maragal\_1\_lasso} & 60 & 60 & 322 & 32 &  \winner  9 &   15 &   10 &  \winner  5.02e-05 &   6.17e-05 &   0.000325 &  \winner  0.000452 &   0.000925 &   0.00325\\ 
\sc{NYPA\_Maragal\_1\_huber} & 96 & 110 & 394 & 32 &   6 &   10 &  \winner  5 &  \winner  5.97e-05 &   8.64e-05 &   0.000603 &  \winner  0.000358 &   0.000864 &   0.00301\\ 
\sc{HB\_ash85\_lasso} & 255 & 255 & 948 & 85 &  \winner  8 &   18 &   15 &  \winner  0.00018 &   0.000325 &   0.00169 &  \winner  0.00144 &   0.00585 &   0.0253\\ 
\sc{HB\_ash85\_huber} & 255 & 340 & 948 & 85 &  \winner  6 &  \winner  6 &   12 &  \winner  0.000297 &   0.000344 &   0.00199 &  \winner  0.00178 &   0.00206 &   0.0238\\ 
\sc{HB\_ash219\_lasso} & 389 & 389 & 997 & 219 &  \winner  8 &   18 &   16 &  \winner  0.000247 &   0.000354 &   0.00255 &  \winner  0.00198 &   0.00637 &   0.0408\\ 
\sc{HB\_ash331\_lasso} & 539 & 539 & 1409 & 331 &  \winner  7 &   18 &   21 &   0.000503 &  \winner  0.0005 &   0.00402 &  \winner  0.00352 &   0.009 &   0.0845\\ 
\sc{HB\_abb313\_lasso} & 665 & 665 & 2574 & 313 &  \winner  9 &   22 &   - &  \winner  0.000607 &   0.000686 &   - &  \winner  0.00546 &   0.0151 &   -\\ 
\sc{HB\_ash219\_huber} & 657 & 742 & 1533 & 219 &  \winner  9 &   16 &   16 &  \winner  0.0003 &   0.000711 &   0.00417 &  \winner  0.0027 &   0.0114 &   0.0667\\
\end{longtable}

\end{landscape}

\captionsetup{labelfont=bf}
\centering
\renewcommand{\detailtablecaption}{\bf Solve times and iteration counts for the NETLIB Feasible LP problem set}
\begin{landscape}
\scriptsize
\begin{longtable}{lcccc||ccc||ccc||ccc||}
\caption{\detailtablecaption}
\\
 & &  & & & \multicolumn{3}{c||}{\underline{iterations}}& \multicolumn{3}{c||}{\underline{time per iteration(s)}}& \multicolumn{3}{c||}{\underline{total time (s)}}\\[2ex] 
Problem & vars. & cons. & nnz(A) & nnz(P)  & ClarabelRs & ECOS & Mosek & ClarabelRs & ECOS & Mosek & ClarabelRs & ECOS & Mosek\\[1ex]
\hline
\endhead
\sc{afiro} & 78 & 51 & 153 & 0 &  \winner  8 &   9 &   11 &   3.91e-05 &  \winner  3.05e-05 &   0.000263 &   0.000312 &  \winner  0.000275 &   0.00289\\ 
\sc{kb2} & 120 & 68 & 390 & 0 &  \winner  16 &  \winner  16 &   28 &   6.83e-05 &  \winner  6.32e-05 &   0.000208 &   0.00109 &  \winner  0.00101 &   0.00583\\ 
\sc{sc50a} & 128 & 78 & 238 & 0 &  \winner  9 &   11 &   11 &   6e-05 &  \winner  4.84e-05 &   0.000276 &   0.00054 &  \winner  0.000532 &   0.00304\\ 
\sc{sc50b} & 128 & 78 & 226 & 0 &  \winner  8 &   10 &   9 &   6.27e-05 &  \winner  4.58e-05 &   0.000305 &   0.000502 &  \winner  0.000458 &   0.00274\\ 
\sc{blend} & 188 & 114 & 636 & 0 &  \winner  11 &   12 &   14 &   0.000115 &  \winner  0.000102 &   0.00102 &   0.00126 &  \winner  0.00122 &   0.0143\\ 
\sc{adlittle} & 194 & 138 & 562 & 0 &  \winner  12 &  \winner  12 &   16 &   0.000115 &  \winner  9.69e-05 &   0.000976 &   0.00137 &  \winner  0.00116 &   0.0156\\ 
\sc{share2b} & 258 & 162 & 939 & 0 &  \winner  11 &   13 &   13 &   0.000166 &  \winner  0.000143 &   0.00128 &  \winner  0.00182 &   0.00186 &   0.0166\\ 
\sc{sc105} & 268 & 163 & 503 & 0 &  \winner  10 &   12 &   12 &   0.00012 &  \winner  0.000103 &   0.000518 &  \winner  0.0012 &   0.00123 &   0.00622\\
\end{longtable}

\end{landscape}

\captionsetup{labelfont=bf}
\centering
\renewcommand{\detailtablecaption}{\bf Solve times and iteration counts for the NETLIB Infeasible LP problem set}
\begin{landscape}
\scriptsize
\begin{longtable}{lcccc||ccc||ccc||ccc||}
\caption{\detailtablecaption}
\\
 & &  & & & \multicolumn{3}{c||}{\underline{iterations}}& \multicolumn{3}{c||}{\underline{time per iteration(s)}}& \multicolumn{3}{c||}{\underline{total time (s)}}\\[2ex] 
Problem & vars. & cons. & nnz(A) & nnz(P)  & ClarabelRs & ECOS & Mosek & ClarabelRs & ECOS & Mosek & ClarabelRs & ECOS & Mosek\\[1ex]
\hline
\endhead
\sc{lpi\_itest2} & 22 & 13 & 39 & 0 &   6 &   6 &  \winner  4 &   1.63e-05 &  \winner  1.2e-05 &   0.00019 &   9.75e-05 &  \winner  7.23e-05 &   0.00076\\ 
\sc{lpi\_galenet} & 30 & 14 & 44 & 0 &   6 &   5 &  \winner  4 &   2.2e-05 &  \winner  1.5e-05 &   0.000187 &   0.000132 &  \winner  7.5e-05 &   0.000749\\ 
\sc{lpi\_itest6} & 28 & 17 & 46 & 0 &   5 &   8 &  \winner  4 &   2.42e-05 &  \winner  1.67e-05 &   0.00024 &  \winner  0.000121 &   0.000134 &   0.000961\\ 
\sc{lpi\_bgprtr} & 60 & 40 & 110 & 0 &   8 &   9 &  \winner  7 &   3.33e-05 &  \winner  2.84e-05 &   0.000262 &   0.000266 &  \winner  0.000256 &   0.00183\\ 
\sc{lpi\_woodinfe} & 138 & 89 & 243 & 0 &   7 &   8 &  \winner  6 &   7.6e-05 &  \winner  5.56e-05 &   0.000569 &   0.000532 &  \winner  0.000445 &   0.00342\\ 
\sc{lpi\_klein1} & 162 & 108 & 858 & 0 &  \winner  18 &  \winner  18 &   21 &   0.000178 &  \winner  0.000164 &   0.000433 &   0.00321 &  \winner  0.00296 &   0.00909\\ 
\sc{lpi\_forest6} & 202 & 131 & 382 & 0 &  \winner  8 &   9 &   10 &   9.06e-05 &  \winner  8.71e-05 &   0.000913 &  \winner  0.000725 &   0.000784 &   0.00913\\ 
\sc{lpi\_ex73a} & 404 & 211 & 668 & 0 &   7 &  \winner  6 &   15 &   0.000207 &  \winner  0.000179 &   0.00115 &   0.00145 &  \winner  0.00107 &   0.0172\\
\end{longtable}

\end{landscape}

\captionsetup{labelfont=bf}
\centering
\renewcommand{\detailtablecaption}{\bf Solve times and iteration counts for the LP Optimal Power Flow problem set}
\begin{landscape}
\scriptsize
\begin{longtable}{lcccc||ccc||ccc||ccc||}
\caption{\detailtablecaption}
\\
 & &  & & & \multicolumn{3}{c||}{\underline{iterations}}& \multicolumn{3}{c||}{\underline{time per iteration(s)}}& \multicolumn{3}{c||}{\underline{total time (s)}}\\[2ex] 
Problem & vars. & cons. & nnz(A) & nnz(P)  & ClarabelRs & ECOS & Mosek & ClarabelRs & ECOS & Mosek & ClarabelRs & ECOS & Mosek\\[1ex]
\hline
\endhead
\sc{case3\_lmbd} & 25 & 9 & 43 & 2 &  \winner  8 &   18 &   - &   1.55e-05 &  \winner  1.27e-05 &   - &  \winner  0.000124 &   0.000229 &   -\\ 
\sc{case3\_lmbd\_\_api} & 25 & 9 & 43 & 2 &  \winner  8 &   16 &   - &  \winner  1.53e-05 &   1.58e-05 &   - &  \winner  0.000122 &   0.000254 &   -\\ 
\sc{case3\_lmbd\_\_sad} & 25 & 9 & 43 & 2 &  \winner  8 &   18 &   - &  \winner  1.51e-05 &   1.59e-05 &   - &  \winner  0.000121 &   0.000286 &   -\\ 
\sc{case5\_pjm} & 46 & 16 & 82 & 0 &  \winner  7 &   8 &   8 &   2.61e-05 &  \winner  2.52e-05 &   0.000194 &  \winner  0.000183 &   0.000202 &   0.00155\\ 
\sc{case5\_pjm\_\_api} & 46 & 16 & 82 & 0 &  \winner  9 &  \winner  9 &   10 &  \winner  2.49e-05 &   3.01e-05 &   0.000186 &  \winner  0.000224 &   0.000271 &   0.00186\\ 
\sc{case14\_ieee} & 125 & 39 & 236 & 0 &  \winner  7 &  \winner  7 &  \winner  7 &   6.64e-05 &  \winner  5.97e-05 &   0.000326 &   0.000465 &  \winner  0.000418 &   0.00228\\ 
\sc{case14\_ieee\_\_api} & 125 & 39 & 236 & 0 &  \winner  11 &  \winner  11 &  \winner  11 &  \winner  5.6e-05 &   6.03e-05 &   0.000248 &  \winner  0.000616 &   0.000663 &   0.00273\\ 
\sc{case30\_as} & 248 & 77 & 470 & 6 &  \winner  8 &   16 &   15 &  \winner  9.86e-05 &   0.000131 &   0.000305 &  \winner  0.000789 &   0.0021 &   0.00457\\
\end{longtable}

\end{landscape}

\captionsetup{labelfont=bf}
\centering
\renewcommand{\detailtablecaption}{\bf Solve times and iteration counts for the SOCP Optimal Power Flow problem set}
\begin{landscape}
\scriptsize
\begin{longtable}{lcccc||ccc||ccc||ccc||}
\caption{\detailtablecaption}
\\
 & &  & & & \multicolumn{3}{c||}{\underline{iterations}}& \multicolumn{3}{c||}{\underline{time per iteration(s)}}& \multicolumn{3}{c||}{\underline{total time (s)}}\\[2ex] 
Problem & vars. & cons. & nnz(A) & nnz(P)  & ClarabelRs & ECOS & Mosek & ClarabelRs & ECOS & Mosek & ClarabelRs & ECOS & Mosek\\[1ex]
\hline
\endhead
\sc{case3\_lmbd} & 124 & 29 & 190 & 0 &  \winner  13 &   15 &   15 &  \winner  5.42e-05 &   5.95e-05 &   0.000252 &  \winner  0.000705 &   0.000893 &   0.00378\\ 
\sc{case3\_lmbd\_\_api} & 124 & 29 & 190 & 0 &  \winner  13 &   14 &   14 &  \winner  5.68e-05 &   6.21e-05 &   0.000255 &  \winner  0.000739 &   0.00087 &   0.00357\\ 
\sc{case3\_lmbd\_\_sad} & 124 & 29 & 190 & 0 &  \winner  13 &   16 &   14 &  \winner  5.3e-05 &   5.95e-05 &   0.00026 &  \winner  0.000688 &   0.000952 &   0.00364\\ 
\sc{case5\_pjm} & 220 & 51 & 352 & 0 &  \winner  16 &   19 &   27 &  \winner  9.52e-05 &   0.000116 &   0.000342 &  \winner  0.00152 &   0.0022 &   0.00924\\ 
\sc{case5\_pjm\_\_api} & 220 & 51 & 352 & 0 &  \winner  17 &   20 &   29 &  \winner  0.000103 &   0.000123 &   0.000343 &  \winner  0.00175 &   0.00247 &   0.00996\\ 
\sc{case5\_pjm\_\_sad} & 220 & 51 & 352 & 0 &  \winner  14 &   18 &   22 &  \winner  0.0001 &   0.000135 &   0.00036 &  \winner  0.0014 &   0.00243 &   0.00793\\ 
\sc{case14\_ieee} & 676 & 144 & 1069 & 0 &  \winner  14 &   18 &   28 &  \winner  0.000283 &   0.000359 &   0.00183 &  \winner  0.00396 &   0.00647 &   0.0512\\ 
\sc{case14\_ieee\_\_api} & 676 & 144 & 1069 & 0 &  \winner  16 &   26 &   31 &  \winner  0.000285 &   0.000395 &   0.00187 &  \winner  0.00456 &   0.0103 &   0.0578\\
\end{longtable}

\end{landscape}

\captionsetup{labelfont=bf}
\centering
\renewcommand{\detailtablecaption}{\bf Solve times and iteration counts for the CBLIB Exponential Cone problem set}
\begin{landscape}
\scriptsize
\begin{longtable}{lcccc||ccc||ccc||ccc||}
\caption{\detailtablecaption}
\\
 & &  & & & \multicolumn{3}{c||}{\underline{iterations}}& \multicolumn{3}{c||}{\underline{time per iteration(s)}}& \multicolumn{3}{c||}{\underline{total time (s)}}\\[2ex] 
Problem & vars. & cons. & nnz(A) & nnz(P)  & ClarabelRs & ECOS & Mosek & ClarabelRs & ECOS & Mosek & ClarabelRs & ECOS & Mosek\\[1ex]
\hline
\endhead
\sc{bss1} & 14 & 11 & 23 & 0 &  \winner  8 &   14 &   9 &   1.28e-05 &  \winner  8.13e-06 &   0.000153 &  \winner  0.000102 &   0.000114 &   0.00138\\ 
\sc{demb782} & 14 & 11 & 23 & 0 &  \winner  6 &   14 &   8 &   1.52e-05 &  \winner  7.99e-06 &   0.000154 &  \winner  9.1e-05 &   0.000112 &   0.00123\\ 
\sc{bss2} & 20 & 15 & 33 & 0 &  \winner  8 &   14 &   9 &   1.59e-05 &  \winner  1.07e-05 &   0.000158 &  \winner  0.000127 &   0.000149 &   0.00142\\ 
\sc{demb781} & 26 & 19 & 40 & 0 &   10 &   16 &  \winner  7 &   1.72e-05 &  \winner  1.45e-05 &   0.000173 &  \winner  0.000172 &   0.000232 &   0.00121\\ 
\sc{gptest} & 32 & 24 & 48 & 0 &  \winner  8 &   16 &   9 &   2.12e-05 &  \winner  1.5e-05 &   0.000168 &  \winner  0.00017 &   0.00024 &   0.00152\\ 
\sc{rijc781} & 32 & 24 & 48 & 0 &  \winner  8 &   16 &   9 &   2.15e-05 &  \winner  1.5e-05 &   0.000176 &  \winner  0.000172 &   0.00024 &   0.00158\\ 
\sc{rijc784} & 40 & 29 & 66 & 0 &  \winner  10 &   18 &  \winner  10 &   2.46e-05 &  \winner  2.25e-05 &   0.000186 &  \winner  0.000246 &   0.000404 &   0.00186\\ 
\sc{rijc785} & 44 & 33 & 83 & 0 &   10 &   17 &  \winner  9 &   2.77e-05 &  \winner  2.45e-05 &   0.000201 &  \winner  0.000277 &   0.000416 &   0.00181\\
\end{longtable}

\end{landscape}

\captionsetup{labelfont=bf}
\centering
\renewcommand{\detailtablecaption}{\bf Solve times and iteration counts for the CBLIB Second-Order Cone problem set}
\begin{landscape}
\scriptsize
\begin{longtable}{lcccc||ccc||ccc||ccc||}
\caption{\detailtablecaption}
\\
 & &  & & & \multicolumn{3}{c||}{\underline{iterations}}& \multicolumn{3}{c||}{\underline{time per iteration(s)}}& \multicolumn{3}{c||}{\underline{total time (s)}}\\[2ex] 
Problem & vars. & cons. & nnz(A) & nnz(P)  & ClarabelRs & ECOS & Mosek & ClarabelRs & ECOS & Mosek & ClarabelRs & ECOS & Mosek\\[1ex]
\hline
\endhead
\sc{sambal} & 63 & 31 & 87 & 0 &  \winner  11 &   14 &   13 &  \winner  3.12e-05 &   3.36e-05 &   0.000231 &  \winner  0.000343 &   0.00047 &   0.00301\\ 
\sc{nb} & 2506 & 2383 & 193902 & 0 &   20 &   - &  \winner  12 &   0.0332 &   - &  \winner  0.0113 &   0.664 &   - &  \winner  0.136\\ 
\sc{nb\_L2\_bessel} & 2764 & 2641 & 211458 & 0 &   10 &   17 &  \winner  7 &   0.0381 &   0.0477 &  \winner  0.0145 &   0.381 &   0.81 &  \winner  0.101\\ 
\sc{nb\_L1} & 4091 & 3176 & 195488 & 0 &   16 &   18 &  \winner  10 &   0.0355 &   0.0363 &  \winner  0.0203 &   0.568 &   0.653 &  \winner  0.203\\ 
\sc{nql30} & 6380 & 4501 & 20569 & 0 &   13 &   19 &  \winner  12 &  \winner  0.00795 &   0.00841 &   0.0361 &  \winner  0.103 &   0.16 &   0.434\\ 
\sc{sched\_50\_50\_scaled} & 7503 & 4977 & 32962 & 0 &   19 &   29 &  \winner  18 &  \winner  0.00667 &   0.00704 &   0.0128 &  \winner  0.127 &   0.204 &   0.231\\ 
\sc{sched\_50\_50\_orig} & 7506 & 4979 & 30467 & 0 &   - &   - &  \winner  28 &   - &   - &  \winner  0.0133 &   - &   - &  \winner  0.373\\ 
\sc{chainsing\_1000\_3} & 10984 & 6991 & 18971 & 0 &   14 &   18 &  \winner  10 &  \winner  0.00558 &   0.00642 &   0.0204 &  \winner  0.0781 &   0.116 &   0.204\\
\end{longtable}

\end{landscape}

%

\captionsetup{labelfont=bf}
\centering
\renewcommand{\detailtablecaption}{\bf Solve times and iteration counts for the SDPLIB Semidefinite Programming problem set}
\begin{landscape}
\newcommand{\nochordal}{
    \begin{tabular}{@{}c@{}}ClarabelRs\\(no decomp)\end{tabular}
}

\scriptsize
\begin{longtable}{lcccc||ccc||ccc||ccc||}
\caption{\detailtablecaption}
\\
 & &  & & & \multicolumn{3}{c||}{\underline{iterations}}& \multicolumn{3}{c||}{\underline{time per iteration(s)}}& \multicolumn{3}{c||}{\underline{total time (s)}}\\[2ex] 
Problem & vars. & cons. & nnz(A) & nnz(P)  & ClarabelRs & \nochordal & Mosek & ClarabelRs & \nochordal & Mosek & ClarabelRs & \nochordal & Mosek\\[1ex]
\hline
\endhead
\sc{truss1} & 19 & 6 & 25 & 0 &   11 &   11 &  \winner  10 &  \winner  7.8e-05 &   8.02e-05 &   0.00015 &  \winner  0.000858 &   0.000882 &   0.0015\\ 
\sc{truss4} & 37 & 13 & 52 & 0 &   10 &   10 &  \winner  9 &   0.00014 &  \winner  0.000131 &   0.000295 &   0.0014 &  \winner  0.00131 &   0.00265\\ 
\sc{hinf1} & 50 & 25 & 116 & 0 &   28 &  \winner  26 &   - &   0.000147 &  \winner  0.000134 &   - &   0.00413 &  \winner  0.00348 &   -\\ 
\sc{hinf2} & 65 & 31 & 154 & 0 &   - &  \winner  18 &   - &   - &  \winner  0.000171 &   - &   - &  \winner  0.00308 &   -\\ 
\sc{hinf3} & 65 & 31 & 154 & 0 &   - &  \winner  19 &   - &   - &  \winner  0.000168 &   - &   - &  \winner  0.00319 &   -\\ 
\sc{hinf4} & 65 & 31 & 154 & 0 &   24 &  \winner  22 &   - &   0.000176 &  \winner  0.000162 &   - &   0.00422 &  \winner  0.00356 &   -\\ 
\sc{hinf5} & 65 & 31 & 154 & 0 &  \winner  24 &   - &   - &  \winner  0.000185 &   - &   - &  \winner  0.00443 &   - &   -\\ 
\sc{hinf6} & 65 & 31 & 154 & 0 &  \winner  23 &   24 &   - &   0.000182 &  \winner  0.000168 &   - &   0.00419 &  \winner  0.00403 &   -\\ 
\sc{hinf8} & 65 & 31 & 154 & 0 &   - &  \winner  19 &   - &   - &  \winner  0.000167 &   - &   - &  \winner  0.00318 &   -\\ 
\sc{hinf9} & 65 & 31 & 154 & 0 &  \winner  24 &   - &   - &  \winner  0.000191 &   - &   - &  \winner  0.00459 &   - &   -\\ 
\sc{control1} & 81 & 36 & 375 & 0 &   54 &  \winner  26 &   - &   0.000415 &  \winner  0.000348 &   - &   0.0224 &  \winner  0.00905 &   -\\ 
\sc{truss3} & 89 & 28 & 120 & 0 &   12 &   12 &  \winner  10 &   0.000356 &  \winner  0.000349 &   0.000609 &   0.00427 &  \winner  0.00419 &   0.00609\\ 
\sc{hinf10} & 95 & 57 & 262 & 0 &  \winner  26 &   - &   - &  \winner  0.000253 &   - &   - &  \winner  0.00658 &   - &   -\\ 
\sc{hinf11} & 156 & 101 & 557 & 0 &  \winner  25 &   - &   - &  \winner  0.000465 &   - &   - &  \winner  0.0116 &   - &   -\\ 
\sc{hinf12} & 219 & 158 & 809 & 0 &  \winner  45 &  \winner  45 &   - &   0.00064 &  \winner  0.000636 &   - &   0.0288 &  \winner  0.0286 &   -\\ 
\sc{truss2} & 298 & 124 & 699 & 0 &  \winner  14 &   16 &   - &   0.00121 &  \winner  0.000999 &   - &   0.017 &  \winner  0.016 &   -\\ 
\sc{control2} & 311 & 121 & 2700 & 0 &  \winner  25 &  \winner  25 &   - &   0.00337 &  \winner  0.00333 &   - &   0.0843 &  \winner  0.0833 &   -\\ 
\sc{qap5} & 351 & 136 & 1026 & 0 &  \winner  9 &  \winner  9 &  \winner  9 &   0.00672 &   0.00643 &  \winner  0.0025 &   0.0605 &   0.0578 &  \winner  0.0225\\ 
\sc{truss7} & 451 & 86 & 863 & 0 &   22 &   22 &  \winner  20 &   0.0016 &  \winner  0.00158 &   0.00273 &   0.0353 &  \winner  0.0348 &   0.0547\\ 
\sc{qap6} & 703 & 229 & 1981 & 0 &  \winner  23 &  \winner  23 &   - &  \winner  0.0205 &   0.0221 &   - &  \winner  0.47 &   0.507 &   -\\ 
\sc{control3} & 691 & 256 & 8850 & 0 &  \winner  29 &  \winner  29 &   - &  \winner  0.0129 &   0.0142 &   - &  \winner  0.373 &   0.411 &   -\\ 
\sc{truss6} & 901 & 172 & 1726 & 0 &   25 &   25 &  \winner  22 &   0.00321 &  \winner  0.00312 &   0.00816 &   0.0803 &  \winner  0.0779 &   0.179\\ 
\sc{theta1} & 1275 & 104 & 153 & 0 &   12 &   12 &  \winner  8 &   0.0726 &   0.0707 &  \winner  0.0255 &   0.871 &   0.849 &  \winner  0.204\\ 
\sc{qap7} & 1275 & 358 & 3480 & 0 &  \winner  30 &   - &   - &  \winner  0.0696 &   - &   - &  \winner  2.09 &   - &   -\\ 
\sc{mcp124\_1} & 974 & 662 & 1200 & 0 &   12 &   13 &  \winner  7 &  \winner  0.00396 &   4.74 &   2.09 &  \winner  0.0475 &   61.7 &   14.6\\ 
\sc{control4} & 1221 & 441 & 20700 & 0 &  \winner  32 &  \winner  32 &   - &   0.0393 &  \winner  0.0392 &   - &   1.26 &  \winner  1.26 &   -\\ 
\sc{truss5} & 1816 & 208 & 2823 & 0 &   18 &   18 &  \winner  16 &   0.0336 &   0.0332 &  \winner  0.0099 &   0.605 &   0.597 &  \winner  0.158\\ 
\sc{control5} & 1901 & 676 & 40125 & 0 &  \winner  30 &  \winner  30 &   - &   0.104 &  \winner  0.101 &   - &   3.11 &  \winner  3.03 &   -\\ 
\sc{qap8} & 2145 & 529 & 5697 & 0 &  \winner  32 &  \winner  32 &   - &   0.22 &  \winner  0.214 &   - &   7.03 &  \winner  6.86 &   -\\ 
\sc{control6} & 2731 & 961 & 69000 & 0 &  \winner  36 &  \winner  36 &   - &  \winner  0.223 &   0.231 &   - &  \winner  8.03 &   8.31 &   -\\ 
\sc{control7} & 3711 & 1296 & 109200 & 0 &  \winner  38 &  \winner  38 &   - &   0.449 &  \winner  0.445 &   - &   17.1 &  \winner  16.9 &   -\\ 
\sc{gpp100} & 5050 & 101 & 5150 & 0 &  \winner  26 &  \winner  26 &   - &  \winner  1.66 &   1.68 &   - &  \winner  43.2 &   43.7 &   -\\ 
\sc{mcp100} & 3168 & 2207 & 4314 & 0 &   11 &   11 &  \winner  7 &  \winner  0.0378 &   1.65 &   0.727 &  \winner  0.416 &   18.2 &   5.09\\ 
\sc{theta2} & 5050 & 498 & 597 & 0 &   10 &   10 &  \winner  9 &   1.78 &   1.81 &  \winner  0.706 &   17.8 &   18.1 &  \winner  6.36\\ 
\sc{mcp250\_1} & 3307 & 2289 & 4328 & 0 &   12 &   - &  \winner  9 &  \winner  0.0277 &   - &   94.5 &  \winner  0.333 &   - &    850\\ 
\sc{qap10} & 5151 & 1021 & 13101 & 0 &  \winner  25 &  \winner  25 &   - &   1.74 &  \winner  1.72 &   - &   43.6 &  \winner    43 &   -\\ 
\sc{control8} & 4841 & 1681 & 162600 & 0 &  \winner  40 &  \winner  40 &   - &  \winner  0.844 &   0.864 &   - &  \winner  33.8 &   34.5 &   -\\ 
\sc{truss8} & 6271 & 496 & 8286 & 0 &   - &   - &  \winner  14 &   - &   - &  \winner  0.0606 &   - &   - &  \winner  0.849\\ 
\sc{mcp124\_2} & 4120 & 2837 & 5550 & 0 &   11 &   12 &  \winner  8 &  \winner  0.0852 &   4.77 &   2.04 &  \winner  0.937 &   57.3 &   16.3\\ 
\sc{gpp124\_1} & 7750 & 125 & 7874 & 0 &  \winner  30 &  \winner  30 &   - &   4.74 &  \winner   4.7 &   - &    142 &  \winner   141 &   -\\ 
\sc{gpp124\_2} & 7750 & 125 & 7874 & 0 &   28 &   28 &  \winner  26 &   4.75 &   4.71 &  \winner  1.82 &    133 &    132 &  \winner  47.4\\ 
\sc{gpp124\_3} & 7750 & 125 & 7874 & 0 &   - &   - &  \winner  24 &   - &   - &  \winner  1.83 &   - &   - &  \winner  43.9\\ 
\sc{gpp124\_4} & 7750 & 125 & 7874 & 0 &   33 &   33 &  \winner  28 &   4.64 &    4.6 &  \winner  1.82 &    153 &    152 &  \winner    51\\ 
\sc{control9} & 6121 & 2116 & 231075 & 0 &  \winner  35 &  \winner  35 &   - &   1.49 &  \winner  1.47 &   - &   52.2 &  \winner  51.4 &   -\\ 
\sc{control10} & 7551 & 2601 & 316500 & 0 &  \winner  43 &  \winner  43 &   - &    2.4 &  \winner  2.38 &   - &    103 &  \winner   102 &   -\\ 
\sc{theta3} & 11325 & 1106 & 1255 & 0 &   10 &   10 &  \winner  9 &   15.7 &   13.6 &  \winner  5.84 &    157 &    136 &  \winner  52.6\\ 
\sc{arch2} & 9230 & 6225 & 15522 & 0 &   24 &  \winner  23 &   - &  \winner  0.14 &   19.2 &   - &  \winner  3.36 &    441 &   -\\ 
\sc{arch4} & 9230 & 6225 & 15522 & 0 &  \winner  20 &   23 &   - &  \winner  0.143 &   19.1 &   - &  \winner  2.85 &    440 &   -\\ 
\sc{arch8} & 9230 & 6225 & 15522 & 0 &   24 &  \winner  23 &   - &  \winner  0.139 &   19.5 &   - &  \winner  3.35 &    448 &   -\\ 
\sc{mcp124\_3} & 9530 & 6634 & 13144 & 0 &   14 &   13 &  \winner  8 &  \winner  0.439 &   4.74 &   2.04 &  \winner  6.15 &   61.6 &   16.3\\ 
\sc{theta4} & 20100 & 1949 & 2148 & 0 &  \winner  10 &  \winner  10 &   11 &   63.5 &   63.2 &  \winner  26.2 &    635 &    632 &  \winner   288\\ 
\sc{mcp250\_2} & 13678 & 9449 & 18648 & 0 &   11 &   - &  \winner  8 &  \winner  0.785 &   - &   95.7 &  \winner  8.64 &   - &    766\\ 
\sc{control11} & 14105 & 9296 & 433070 & 0 &  \winner  40 &   45 &   - &  \winner   3.9 &   4.46 &   - &  \winner   156 &    201 &   -\\ 
\sc{arch0} & 13743 & 10172 & 23026 & 0 &  \winner  22 &   26 &   - &  \winner  0.372 &   19.1 &   - &  \winner  8.17 &    496 &   -\\ 
\sc{mcp124\_4} & 14594 & 9546 & 18968 & 0 &   12 &   12 &  \winner  7 &  \winner  1.85 &   4.74 &   2.11 &   22.2 &   56.9 &  \winner  14.8\\ 
\sc{theta5} & 31375 & 3028 & 3277 & 0 &   - &   - &  \winner  10 &   - &   - &  \winner  94.1 &   - &   - &  \winner   941\\ 
\sc{mcp250\_3} & 40617 & 29866 & 59482 & 0 &   13 &   - &  \winner  8 &  \winner  11.1 &   - &   95.4 &  \winner   145 &   - &    763\\ 
\sc{mcp250\_4} & 59628 & 41799 & 83348 & 0 &   14 &   - &  \winner  9 &  \winner  47.4 &   - &   94.4 &  \winner   664 &   - &    850\\ 
\end{longtable}
\end{landscape}

\end{document}